\documentclass[12pt]{amsart}

\usepackage[hmargin=0.8in,height=8.6in]{geometry}
\usepackage{amssymb,amsthm, times}
\usepackage{delarray,verbatim}
\usepackage{ifpdf}
\ifpdf
\usepackage[pdftex]{graphicx}
 \else
\usepackage[dvips]{graphicx}
 \fi

\usepackage{bm}

\usepackage{ifpdf}
\usepackage{color,soul}
\usepackage[bordercolor=white,backgroundcolor=gray!30,linecolor=black,colorinlistoftodos]{todonotes}

\usepackage{xcolor}
\usepackage{mdframed}
{\begin{mdframed}[backgroundcolor=yellow]\begin{remark}}%
		{\end{remark}\end{mdframed}}
\definecolor{webgreen}{rgb}{0,.5,0}
\definecolor{webbrown}{rgb}{.8,0,0}
\definecolor{emphcolor}{rgb}{0.95,0.95,0.95}
\usepackage{hyperref}
\hypersetup{%
          colorlinks=true,
          linkcolor=webbrown,
          filecolor=webbrown,
          citecolor=webgreen,
          breaklinks=true}
\ifpdf \hypersetup{pdftex,
            pdfstartview=FitH, 
            bookmarksopen=true,
            bookmarksnumbered=true
} \else \hypersetup{dvips} \fi

\newcommand {\ud}{{\rm d}}

\usepackage{color}

\numberwithin{equation}{section}

\newtheorem{theorem}{Theorem}[section]
\newtheorem{proposition}{Proposition}[section]
\newtheorem{corollary}{Corollary}[section]
\newtheorem{remark}{Remark}[section]
\newtheorem{lemma}{Lemma}[section]

\newtheorem{assump}{Assumption}[section]

\numberwithin{remark}{section} \numberwithin{proposition}{section}
\numberwithin{corollary}{section}
\newcommand {\R}{\mathbb{R}}

\newcommand {\p}{\mathbb{P}}

\newcommand {\E}{\mathbb{E}}

\newcommand{\diff}{{\rm d}}

\newcommand{\lev}{L\'{e}vy }

\newcommand{\e}{\mathbb{E}}

\title{Optimality of multi-refraction  dividend strategies in the dual model}
\thanks{I. Czarna is partially supported by the National Science Centre Grant No. 2015/19/D/ST1/01182. 
	J. L. P\'erez  is  supported  by  CONACYT,  
project no.\ 241195. K. Yamazaki is supported by MEXT KAKENHI grant no.\ 17K05377.
 }
\date{\today.  }

\author[I. Czarna]{Irmina Czarna$^*$}
\address{$*$\,Faculty of Pure and Applied Mathematics, Wroc\l aw University
of Science and Technology, ul. Wybrze\.ze Wyspia\'nskiego 27, 50-370 Wroc\l aw, Poland. Email: irmina.czarna@pwr.edu.pl}
\author[J. L. P\'erez]{Jos\'e-Luis P\'erez$^{**}$}
\address{$**$\, Department of Probability and Statistics, Centro de Investigaci\'on en Matem\'aticas A.C. Calle Jalisco s/n. C.P. 36240, 
Guanajuato, Mexico. Email: jluis.garmendia@cimat.mx }
\author[K. Yamazaki]{Kazutoshi Yamazaki$^{\dag \dag}$}
\address{$\dag \dag$\, Department of Mathematics,
Faculty of Engineering Science, Kansai University, 3-3-35 Yamate-cho, Suita-shi, Osaka 564-8680, Japan. Email: kyamazak@kansai-u.ac.jp   }

\begin{document}
	\maketitle

\begin{abstract} 
We consider the \emph{multi-refraction strategies} in two equivalent versions of the optimal dividend problem in the dual (spectrally positive L\'evy) model. The first problem is a variant of the bail-out case where both dividend payments and capital injections must be absolutely continuous with respect to the Lebesgue measure. The second is an extension of Avanzi et al.\ \cite{APWY} where a strategy is a combination of two absolutely continuous dividend payments with different upper bounds and  different transaction costs.  In both problems, it is shown to be optimal to refract the process at two thresholds, with the optimally controlled process being the \emph{multi-refracted \lev process} recently studied by Czarna et al.\ \cite{CPRY}.  The optimal strategy and the value function are succinctly written in terms of a version of the scale function.
 Numerical results are also given.
\\
\noindent \small{\noindent  AMS 2010 Subject Classifications: 60G51, 93E20, 91B30 
}\\
\end{abstract}

\section{Introduction}

In this paper, we study two equivalent optimal dividend problems. The first problem is an extension of the bail-out model where capital can be injected so as to reduce the risk of ruin. Contrary to the classical case, 
both dividend and capital injection strategies must be absolutely continuous with respect to the Lebesgue measure, with their densities required to be bounded by some fixed constants. Because of the restriction on capital injections, ruin may occur.
The second problem pursues an optimal pair of two absolutely continuous dividend strategies with different upper bounds on the densities and with different proportional transaction costs. 
By incorporating the terminal payoff/cost at ruin, these two problems can easily be shown to be equivalent.

Regarding the first problem, the classical bail-out problem has been solved by Avram et al.\ \cite{APP2007} and Bayraktar et al.\ \cite{BKY} for the spectrally negative and positive (dual) models, respectively. 
Without the absolutely continuous assumption on the dividend and capital injection strategies, it is optimal to \emph{reflect} from below at zero and from above at some suitable barrier -- the resulting surplus process becomes the doubly reflected 
\lev process of Pistorius \cite{P2003}.  Recently, the case the dividend (and not the capital injection) strategy is absolutely continuous has been solved by P\'erez et al. \cite{PYY} and  P\'erez and Yamazaki \cite{YP_RR_dual}, 
for the spectrally negative and positive models, respectively. They showed that it is optimal to reflect from below at zero and refract at a suitably chosen upper barrier, with the resulting process being the refracted-reflected \lev process.

On the other hand, the second problem can be viewed as an extension to Avanzi et al.\ \cite{APWY}, where they considered the case a dividend 
strategy is composed of a singular and absolutely continuous parts with different associated transaction costs. 
They showed the optimality of a strategy that refracts the process at the lower threshold and reflects it at the upper barrier, similarly to the case in \cite{YP_RR_dual}.

The objective of this current paper is to show, for both problems, the optimality of what we call the \emph{multi-refraction strategies}, 
that refract the underlying process at two different thresholds, say $a^* \leq b^*$, to be suitably chosen. 
More precisely, for the first problem, we aim to show that it is optimal to pay dividends at the maximum possible rate whenever the surplus is above $b^*$ while injecting capital at the maximum rate whenever it is below $a^*$.  For the second problem, it is optimal to activate fully one of the absolutely continuous strategies whenever the surplus is above $a^*$ while activating fully the other one as well if it is above $b^*$.

We focus on the dual model, where the underlying (uncontrolled) surplus process follows a spectrally positive \lev process. 
See Avanzi et al.\ \cite{AGS2007,AGS2008} for a more detailed motivation of this model.  

Recently, Czarna et al.\ \cite{CPRY} studied a generalization of the refracted \lev process of Kyprianou and Loeffen \cite{KL} with multiple refraction thresholds.  
Under a multi-refraction strategy, the resulting surplus process becomes precisely the one studied in \cite{CPRY}, and hence 
the fluctuation identities obtained there can be directly used.  The expected net present value (NPV) of dividends/capital injections for our problems can be written efficiently using the generalization of the scale function.


In order to solve the problems, we take the following steps:
\begin{enumerate}

\item Focusing on the set of multi-refraction strategies, we shall first select the two refraction thresholds $a^*$ and $b^*$, so that the corresponding candidate value function, say $v_{a^*, b^*}$, becomes \emph{smoother} at these thresholds. 
 More precisely, we choose these so that it is continuously differentiable (resp.\ twice continuously differentiable) for the case the underlying process has paths of bounded (resp.\ unbounded) variation.  
These conditions at  $a^*$ and $b^*$ give us two equations, which we call $\mathbf{C}_a$ and $\mathbf{C}_b$, respectively, in this paper.  We select $(a^*, b^*)$ so that $\mathbf{C}_a$ and $\mathbf{C}_b$ hold if $a^* > 0$ and $b^* > 0$, respectively.



\item We then show the optimality of the selected multi-refraction strategy. Toward this end, we obtain sufficient conditions for optimality (verification lemma) and show that 
 the candidate value function $v_{a^*, b^*}$ indeed satisfies these.

\end{enumerate}



We also conduct numerical experiments so as to confirm the obtained analytical results.  We use  the phase-type \lev process of \cite{Asmussen_2004} whose scale function admits the form of a linear combination of (complex) exponentials (see \cite{Egami_Yamazaki_2010_2}), and hence the optimal thresholds $a^*$ and $b^*$ and the value function can be computed instantaneously.  We illustrate the computation procedures and also confirm the optimality of the multi-refraction strategies.  Additionally, we analyze and confirm the convergence to other existing versions of the optimal dividend problems.

The rest of the paper is organized as follows. Section \ref{section_preliminary} formulates 
the two problems we study in this paper. 
We define the multi-refraction strategy and compute the expected NPV under this strategy in Section \ref{section_multi_refraction_strategies}.
In Section \ref{section_smooth_fit}, we choose candidate optimal thresholds
$(a^*,b^*)$, whose optimality is confirmed in Section \ref{section_optimality}.
  We conclude the paper with numerical results in Section \ref{numerical_section}.  Some proofs are deferred to the Appendix.

\section{Mathematical models} \label{section_preliminary}

In this section, we define the two problems considered in this paper.

\subsection{Problem 1 (bail-out model with absolutely continuous assumptions)} \label{section_dividend}

Let $Y=(Y(t); t\geq 0)$ be a L\'evy process defined on a  probability space $(\Omega, \mathcal{F}, \p)$, 
modeling the surplus of a company in the absence of control. We assume that it is \emph{spectrally positive} 
or equivalently it has no negative jumps and is not a subordinator.
For $x\in \R$, we denote by $\p_x$ the law of $Y$ 
when it starts at $x$ and write for convenience  $\p$ in place of $\p_0$. 

An admissible strategy  is a pair $\pi := \left(L^{\pi}(t), R^{\pi}(t); t \geq 0 \right)$ of nondecreasing, right-continuous, 
and adapted processes (with respect to the filtration generated by $Y$) such that $L^{\pi}(0) = R^{\pi}(0) = 0$ where $L^{\pi}$ 
is the cumulative amount of dividends and $R^{\pi}$ is that of injected capital. In addition,  for $\delta_1,\delta_2 > 0$ fixed, we require that $L^\pi$ and $R^\pi$ are absolutely continuous with respect to the Lebesgue measure of the forms $L^\pi(t) = \int_0^t \ell^{\pi}(s) \diff s$ and $R^\pi(t) = \int_0^t r^{\pi}(s) \diff s$, $t \geq 0$, with $\ell^{\pi}$ and $r^{\pi}$ restricted to take values in, respectively, $[0,\delta_1]$ and $[0,\delta_2]$ uniformly in time.
We will denote by $V^{\pi}$ the controlled surplus process associated with the strategy $\pi$:
\begin{align}
V^{\pi}(t)=Y(t)-L^{\pi}(t)+R^{\pi}(t), \quad t \geq 0. \label{V_pi_decomp}
\end{align}

Assuming that $\beta > 1$ is the cost per unit injected capital, $\rho \in \R$ the terminal  payoff (if $\rho\geq  0$)/penalty (if $\rho\leq  0$) 
at ruin, and $q > 0$ the discount factor, we want to maximize
\begin{align}
	v_{\pi} (x) := \mathbb{E}_x \left( \int_0^{\kappa_0^{\pi}} e^{-q t} \ell^{\pi}(t)  
	\diff t - \beta \int_0^{\kappa_0^{\pi}} e^{-q t} r^{\pi}(t) \diff t + \rho e^{- q \kappa_0^{\pi}}  \right), \quad x \geq 0, \label{v_pi_1}
\end{align}
where
\begin{align}
\kappa_0^{\pi}:=\inf\{t>0: V^{\pi}(t)<0\}. \label{kappa_0_pi}
\end{align}
Hence the problem is to compute
\begin{equation*}
	v(x):=\sup_{\pi \in \mathcal{A}}v_{\pi}(x), \quad x \geq 0,
\end{equation*}
where $\mathcal{A}$ is the set of all admissible strategies  that satisfy the constraints described above.

For this problem, we assume the following. Note that, without this condition, one can always avoid ruin by injecting enough capital.
%
\begin{assump} \label{assump_subordinator} We assume that process $( Y(t) + \delta_2 t; t \geq 0 )$ is not a subordinator.

\end{assump}


\subsection{Problem 2} \label{subsection_problem2} We consider a variant of the problem studied 
in  Avanzi et al.\ \cite{APWY}.  
Let $\tilde{Y}$ be a spectrally positive \lev process, and, as in Problem 1, let $\E_x$ be the expectation under which $\tilde{Y}(0) = x$.

A strategy  is a pair of two dividend strategies $\vartheta := (L^{\vartheta}(t) =  \int_0^t \ell^{\vartheta}(s) \diff s, \tilde{R}^{\vartheta}(t) = \int_0^t \tilde{r}^{\vartheta}(s) \diff s; t \geq 0 )$ 
that satisfy the same properties required for $\mathcal{A}$ of Problem 1, with the corresponding controlled surplus process
\[
\tilde{V}^{\vartheta}(t)=\tilde{Y}(t)-L^{\vartheta}(t)- \tilde{R}^{\vartheta}(t), \quad t \geq 0.
\] 
  
For $\tilde{R}^\vartheta$, the unit dividend rate is $\beta > 1$ while, for $L^\vartheta$, proportional transaction costs are incurred and the unit dividend rate is $1$.
With $\tilde{\rho} \in \R$ the terminal payoff/penalty at ruin, 
one wants to maximize
\begin{align}
\tilde{v}_{\vartheta} (x) := \mathbb{E}_x \left(   \int_0^{\tilde{\kappa}_0^{\vartheta}} e^{-q t} \ell^{\vartheta}(t) \diff t + \beta \int_0^{\tilde{\kappa}_0^{\vartheta}} e^{-q t} \tilde{r}^{\vartheta}(t) \diff t  + \tilde{\rho} e^{-q \tilde{\kappa}_0^{\vartheta}}    \right), \quad x \geq 0, \label{v_tilde}
\end{align}
where
$\tilde{\kappa}_0^{\vartheta} := \inf \{ t > 0: \tilde{V}^\vartheta(t) < 0 \}$.
%
%
%
%

This can be easily transformed to Problem 1.  To see this,  setting the process $Y$ and a strategy $\pi = (\int_0^t  \ell^\pi(s)\diff s, \int_0^t   r^\pi(s) \diff s; t \geq 0) \in 
\mathcal{A}$ with
\begin{align*}
Y(t)  := \tilde{Y}(t) - \delta_2 t, \quad  r^\pi(t):= \delta_2 - \tilde{r}^\vartheta(t) \quad \textrm{and} \quad  \ell^\pi(t):=\ell^\vartheta(t), \quad t \geq 0,
\end{align*}
we have  $\tilde{V}^\vartheta(t) = (Y(t) + \delta_2 t) - \int_0^t \ell^\pi(s) \diff s + 
(\int_0^t r^\pi(s) \diff s - \delta_2 t) = V^\pi(t) $, $t \geq 0$ (as in \eqref{V_pi_decomp}), and hence $\tilde{\kappa}_0^\vartheta = \kappa_0^\pi$ 
(as in \eqref{kappa_0_pi}).
Now for all $x \geq 0$, with $\rho := \tilde{\rho} - \beta \delta_2 / q$, \eqref{v_tilde} becomes 
\begin{align} \label{problem_2_equivalent}
\begin{split}
\tilde{v}_{\vartheta} (x) 
&= \mathbb{E}_x \left(   \int_0^{\kappa_0^{\pi}} e^{-q t} \ell^{\pi}(t)  
\diff t - \beta \int_0^{\kappa_0^{\pi}} e^{-q t} r^\pi(t) \diff t + \rho e^{-q \kappa_0^{\pi}} \right) + 
 \frac {\beta \delta_2} q.
\end{split} 
\end{align}
In other words, the problem reduces to maximizing \eqref{v_pi_1}.  

\vspace{0.5cm}
Because Problems 1 and 2 are equivalent, for the rest of the paper except for the numerical results given in Section \ref{numerical_section}, we focus on Problem 1.

\section{Multi-refraction strategies} \label{section_multi_refraction_strategies}



Our objective of this paper is to show the optimality of the multi-refraction strategy, say $\pi^{a,b}$, with suitable refraction levels $0 \leq  a\leq b$.

Fix $0 \leq  a\leq b$. Under $\pi^{a,b}$, dividends are paid at the maximal possible rate $\delta_1$ whenever the 
surplus is above $b$ while capital is injected at the maximal possible rate $\delta_2$ whenever it is below $a$. 
In this case the aggregate process $V^{a,b}$ is given by 
\begin{align*}
V^{a,b}(t)=Y(t)-\delta_1\int_0^t1_{\{V^{a,b}(s)\geq b\}} \diff s+\delta_2\int_0^t1_{\{V^{a,b}(s) < a\}} \diff s\qquad\text{ $t\geq 0$.}
\end{align*}
In order to see this is a well-defined process, note that, for all $t \geq 0$, 
\begin{align} \label{V_a_b_SDE}
\begin{split}
-V^{a,b}(t)
=X_0(t)-\delta_1\int_0^t1_{\{-V^{a,b}(s)> -b\}} \diff s-\delta_2\int_0^t 1_{\{-V^{a,b}(s)> -a\}} \diff s,
\end{split}
\end{align}
where 
\begin{align}
X_0(t):=-Y(t)+\delta_1 t, \quad t \geq 0. \label{X_def}
\end{align}
In addition, for $t \geq 0$, let 
\begin{align} \label{X_k_def}
\begin{split}
X_1(t) &:= X_0(t) - \delta_1 t = - Y(t), \\
X_2(t) &:= X_0(t) - (\delta_1+\delta_2) t =  - Y(t) - \delta_2 t,
\end{split}
\end{align}
which by Assumption \ref{assump_subordinator} are spectrally negative \lev processes (that are not the negative of subordinators).
%
\par The process $-V^{a,b}$ is a spectrally negative multi-refracted L\'evy process of \cite{CPRY} driven by the process $X_0$ with refraction 
thresholds $-b \leq -a \leq 0$, which is a unique strong solution to  \eqref{V_a_b_SDE}. 
It behaves like $X_0$ on $(-\infty, -b)$, like $X_1$ on $(-b,-a)$, and like $X_2$ on $(-a,\infty)$. 

 It is clear that the strategy $\pi_{a,b}$ is admissible and its expected NPV of the total payoff is given by 
\begin{align} \label{v_pi}
	v_{a,b}(x) := \mathbb{E}_x \left( \int_0^{\kappa_0^{a,b}} e^{-q t} \diff L^{a,b}(t) - \beta \int_0^{\kappa_0^{a,b}} e^{-q t} \diff R^{a,b}(t)+ \rho e^{- q \kappa_0^{a,b}} \right), \quad x \geq 0,
\end{align}
where $\diff L^{a,b}(t)=\delta_11_{\{V^{a,b}(t) \geq b\}} \diff t$ and $\diff R^{a,b}(t)=\delta_21_{\{V^{a,b}(t)< a\}}\diff t$ for $t>0$, and
\[
\kappa_0^{a,b}:=\inf\{t>0: V^{a,b}(t)<0\}.
\]



\subsection{Scale functions}  \label{section_scale_functions}

Using the results in \cite{CPRY}, the expected NPV \eqref{v_pi} under the multi-refraction strategy can be written in terms of the scale functions $W_k^{(q)}$ and $Z_k^{(q)}$ of the spectrally negative \lev processes 
$X_k$  for $k=0,1,2$ defined in 
\eqref{X_def} and \eqref{X_k_def}. 

Fix $q > 0$. Define the Laplace exponent of $X_0$ by  $\psi_0:[0,\infty) \to \R$ such that
\[
{\rm e}^{\psi_0(\lambda)t}:=
\e\big[{\rm e}^{\lambda X_0 (t)}\big], \qquad t, \lambda\ge 0,
\]
with its  \emph{L\'evy-Khintchine representation}, 
\begin{equation}\label{lk}
	\psi_0(\lambda) = \gamma_0 \lambda+\frac{\sigma^2}{2}\lambda^2+\int_{(0, \infty)}\big({\rm e}^{-\lambda z}-1+\lambda z\mathbf{1}_{\{z<1\}}\big)\Pi(\ud z), \quad \lambda \geq 0,
\end{equation}
where $\gamma_0 \in \R$, $\sigma\ge 0$, and $\Pi$ is a  measure on $(0, \infty)$ 
 that satisfies
\[
\int_{(0, \infty)}(1\land z^2)\Pi(\ud z)<\infty.
\]
We also define the Laplace exponents of $X_k$ for $k=1,2$ by
\begin{align*}
\psi_k(\lambda) =  \psi_0(\lambda) - \sum_{i=1}^k \delta_i \lambda = \gamma_k \lambda+\frac{\sigma^2}{2}\lambda^2+\int_{(0, \infty)}\big({\rm e}^{-\lambda z}-1+\lambda z\mathbf{1}_{\{z<1\}}\big)\Pi(\ud z), \quad \lambda \geq 0,
\end{align*}
where $\gamma_k := \gamma_0 - \sum_{i=1}^k \delta_i $.

The processes $X_k$, for $k=0,1,2$, have paths of bounded variation if and only if $\sigma=0$ and $\int_{(0,1)} z\Pi(\mathrm{d} z) < \infty$; in this case, we can write 
\begin{equation}
	X_k(t)=c_k t-S(t), \,\,\qquad t\geq 0, \quad k = 0,1,2,  \notag
	\label{BVSNLP}
\end{equation}
where 
\begin{align}
	c_k:=\gamma_k+\int_{(0,1)} z \Pi(\mathrm{d} z) \label{def_drift_finite_var}
\end{align}
and $(S(t); t\geq0)$ is a driftless subordinator.

\begin{remark} \label{assump_drift}For the case of bounded variation, Assumption \ref{assump_subordinator} is  equivalent to $c_2 = c_0 - \delta_1 - \delta_2 > 0$. 
\end{remark}



Fix $q > 0$ and $k = 0,1,2$. The $q$-scale function of the process $X_k$ is defined as the continuous function on $[0,\infty)$ with its Laplace transform
\begin{equation}\label{def_scale}
\int_0^{\infty} \mathrm{e}^{- \lambda y} W^{(q)}_k (y) \mathrm{d}y = \frac{1}{\psi_k(\lambda) - q} , \quad \text{for $\lambda > \Phi_k(q)$}, 
\end{equation}
where $\Phi_k \colon [0,\infty) \to [0,\infty)$ is the right-inverse given by $\Phi_k(q) := \sup \{ \lambda \geq 0 : \psi_k(\lambda) = q\}$.
This function is unique, positive and strictly increasing  and is further continuous for $x\geq0$. We extend $W_k^{(q)}$ to the whole real line by setting $W_k^{(q)}(x)=0$ for $x<0$. 
We also define, for $x \in \R$, 
\begin{align}\label{eq:zqscale}
\begin{split}
\overline{W}_k^{(q)}(x) &:= \int_0^x W^{(q)}_k (y)\mathrm d y, \\
Z^{(q)}_k(x) &:= 1 + q \overline{W}^{(q)}_k(x) .
\end{split}
\end{align}
These scale functions are related by the following equalities: for $k=1,2$,  
\begin{align}\label{RLqp}
	&\sum_{i=1}^k\delta_i \int_0^xW^{(q)}_k(x-y)  W_0^{(q)}(y) \ud y=\overline{W}_k^{(q)}(x)-\overline{W}_0^{(q)}(x), \quad  x \in \R \; \textrm{and} \; q >0,
\end{align}
which can be proven by showing that the Laplace transforms on both sides are equal.

Regarding their asymptotic values as $x \downarrow 0$, we have, as in Lemmas 3.1 and 3.2 of \cite{KKR}, 
\begin{align}\label{eq:Wqp0}
	\begin{split}
		W^{(q)}_k (0) &= \left\{ \begin{array}{ll} 0 & \textrm{if $X_k$ is of unbounded
				variation,} \\ 
				c_k^{-1}
				& \textrm{if $X_k$ is of bounded variation,}
		\end{array} \right.  
	\end{split} \\ \label{W_zero_derivative}
	\begin{split}
		W^{(q)\prime}_k (0 +) &:= \lim_{x \downarrow 0} W^{(q)\prime}_k (x+) =
		\left\{ \begin{array}{ll}  \frac 2 {\sigma^2} & \textrm{if }\sigma > 0, \\
			\infty & \textrm{if }\sigma = 0 \; \textrm{and} \; \Pi(0, \infty) = \infty, \\
			\frac {q + \Pi(0, \infty)} {c_k
			^2} &  \textrm{if }\sigma = 0 \; \textrm{and} \; \Pi(0, \infty) < \infty,
		\end{array} \right. 
	\end{split}
\end{align}
and, as in Lemma 3.3 of \cite{KKR}, 
\begin{align}
	\begin{split}
		e^{-\Phi_k(q) x}W^{(q)}_k (x) \nearrow \psi_{k}'(\Phi_k(q))^{-1} \quad \textrm{as } x \rightarrow \infty.
	\end{split}
	\label{W_q_limit}
\end{align}
 Here and for the rest of the paper, $g'(x+)$ and $g'(x-)$,  for any function $g$, are the right-hand and left-hand derivatives, respectively, at $x$. 


\begin{remark} \label{remark_smoothness} 
	It is known that the right-hand and left-hand derivatives of the scale function always exist for all $x > 0$.
	If $Y$ (and hence $X_k$ for $k=0,1,2$ as well) is of unbounded variation or the \lev measure is atomless, 
	it is known that $W_k^{(q)}$ is $C^1(\R \backslash \{0\})$ for $k = 0,1,2$. 
	For more comprehensive results on the smoothness, see \cite{Chan2011}.
\end{remark}

Finally, the following function will be important in the derivation of the results in the rest of the paper. For $x, b \in\mathbb{R}$
, we define
 	\begin{align} \label{def_l_x_b}
 		l^{(q)}(x; b) &:=   \int^{b}_{x} e^{-\Phi_0(q) z} W_1^{(q)} (z-x) \diff z =   e^{-\Phi_0(q)b} \overline{W}^{(q)}_1(b-x) + \Phi_0(q) \int^{b}_{x} e^{-\Phi_0(q) z} \overline{W}_1^{(q)} (z-x) \diff z,
 	\end{align}
 	where the second equality holds by integration by parts.
		Differentiating \eqref{def_l_x_b}  with respect to $x$, we have, for $x \neq b$,
		\begin{align}\label{l_prime}
		l^{(q)\prime}(x; b) &= - \left( e^{-\Phi_0(q)b} W^{(q)}_1(b-x) + \Phi_0(q) l^{(q)}(x; b)\right)  \leq  0, \\
			l^{(q)\prime \prime}(x-; b) 
			&=   e^{-\Phi_0(q)b} W^{(q)\prime}_1((b-x)+) - \Phi_0(q)  l^{(q)\prime}(x;b). \label{l_prime_double_prime}
		\end{align}

\subsection{Computation of \eqref{v_pi}} 

Using the results obtained in \cite{CPRY}, we present the following explicit form of the function $v_{a,b}(x)$. 
 Its proof is deferred to Appendix \ref{appendix_proof_value_fun_explicit}.

\begin{lemma}\label{value_fun_explicit} For $0 \leq a \leq b$ and $x \geq 0$, we have
\begin{align*}
	v_{a,b}(x)
					&= f_{a,b}^{(q)}(x) - (f_{a,b}^{(q)}(0) - \rho) \frac {g_{a,b}^{(q)}(x)} {g_{a,b}^{(q)}(0)},
		 \end{align*}
		where 
\begin{align} \label{def_f}
	f_{a,b}^{(q)}(x) &:= \delta_1 \Big( \overline{W}_1^{(q)} (b-x)+ \delta_2  \int^{a}_{x} W_2^{(q)} (u-x)  W^{(q)}_1 (b-u)  \diff u \Big)  +\beta\delta_2 \overline{W}^{(q)}_2 (a-x) + \frac {\delta_1} q, \\
\label{u_simplified}
	\begin{split}
	g_{a,b}^{(q)}(x)
	&:= e^{-\Phi_0(q) x} +  \Phi_0(q) \Big[ \delta_1 l^{(q)}(x;b) 
	+	  \delta_2  \int^{a}_{x} W_2^{(q)}(z-x)  \big( e^{-\Phi_0(q) z}-  \delta_1  l^{(q)\prime}(z;b) \big) \diff z \Big].
	\end{split}
\end{align} 
\end{lemma}
\section{Selection of $(a^*, b^*)$}  \label{section_smooth_fit}

In this section, using the smooth fit principle we choose the (candidate) optimal refraction thresholds, 
which we shall call $(a^*, b^*)$. 
In particular, we shall choose their values so that $v_{a^*, b^*}$ becomes continuously differentiable (resp.\ twice continuously differentiable) for the case $Y$ (and $X_k$, $k=0,1,2$) are  of bounded (resp.\ unbounded) variation.


We first differentiate the identities given in Lemma \ref{value_fun_explicit}. 
The next two lemmas follow by direct differentiation and integration by parts, and hence we omit the proofs.
\begin{lemma} \label{lemma_derivatives_f}
Fix $0\leq a\leq b$. For $x\in \R_+\backslash\{a,b\}$, 
	\begin{align*}
		f_{a,b}^{(q)\prime}(x)&= -  \delta_2  \Big( (\beta + \delta_1 W_1^{(q)}(b-a)) W_2^{(q)}(a-x) +  
		\delta_1 \int^{a}_{x} W_2^{(q)} (u-x)  W^{(q) \prime}_1 (b-u)  \diff u  \Big) 1_{\{ x < a\}} \notag\\&  -\delta_1  W_1^{(q)} (b-x), \\ 
	f_{a,b}^{(q)\prime \prime}(x-) &=  \delta_2  \Big( (\beta + \delta_1 W_1^{(q)}(b-a)) W_2^{(q)\prime}((a-x)+)+ \delta_1 W_2^{(q)} (0)  W^{(q) \prime}_1 ((b-x)+) \\
	&\qquad \qquad +  \delta_1 \int^{a}_{x} W_2^{(q)\prime} (u-x)  W^{(q) \prime}_1 (b-u)  \diff u  \Big) 
	1_{\{ x <a \}}  + \delta_1  W_1^{(q) \prime} ((b-x)+). 
\end{align*}
\end{lemma}

\begin{lemma} \label{lemma_derivatives_g}
	Fix $0\leq a\leq b$. For $x \in \R_+ \backslash \{a,b\}$, we have
	\begin{align*}
		\frac {g_{a,b}^{(q)\prime}(x) } {\Phi_0(q)}
		&= - e^{-\Phi_0(q) x} + \delta_1  l^{(q)\prime}(x;b) -\delta_2 1_{\{ x < a \}} \Big[  W^{(q)}_2(a-x) \Big( e^{- \Phi_0(q)a} - \delta_1 l^{(q)\prime}(a;b) \Big)\notag\\& + \int^{a}_{x} W_2^{(q)}(z-x)  \Big( \Phi_0(q) e^{-\Phi_0(q) z} + \delta_1 l^{(q)\prime \prime}(z;b)  \Big) \diff z \Big]\\
				\frac {g_{a,b}^{(q)\prime \prime}(x-) } {\Phi_0(q)}
				&=   \Big( \Phi_0(q) e^{-\Phi_0(q) x} + \delta_1  l^{(q)\prime \prime}(x;b) \Big) \Big( 1 + \delta_2 W_2^{(q)} (0) 1_{\{ x < a \}} \Big)+  \delta_2 1_{\{ x < a \}} \Big[  W^{(q)\prime}_2((a-x)+) \notag\\&\times\Big( e^{- \Phi_0(q)a} - \delta_1 l^{(q)\prime}(a;b) \Big)+ \int^{a}_{x} W_2^{(q)\prime}(z-x)  \Big( \Phi_0(q) e^{-\Phi_0(q) z} + \delta_1 l^{(q)\prime \prime}(z;b)  \Big) \diff z \Big].
			\end{align*}

\end{lemma}
By taking limits in the above identities,
 the following results are immediate.
\begin{corollary} \label{lemma_difference_at_a_b_f}
 Fix $b > a\geq  0$ for the derivative at $b$ and $b > a > 0$ for the derivative at $a$.
(i) We have
\begin{align*}
f_{a,b}^{(q)\prime}(b+) - f_{a,b}^{(q)\prime}(b-) &=  \delta_1 W_1^{(q)} (0) \quad \textrm{and} 
\quad f_{a,b}^{(q)\prime}(a+) - f_{a,b}^{(q)\prime}(a-)  = \delta_2 W^{(q)}_2 (0) (\beta + \delta_1 W_1^{(q)}(b-a)).
\end{align*}
(ii)  In particular, for the case of unbounded variation, we have
\begin{align*}
f_{a,b}^{(q)\prime \prime}(b+) - f_{a,b}^{(q)\prime \prime}(b-) &=  - \delta_1 W_1^{(q) \prime} (0+) \quad \textrm{and} \quad 
f_{a,b}^{(q)\prime \prime}(a+) - f_{a,b}^{(q)\prime \prime}(a-)  = - \delta_2 W_2^{(q)\prime}(0+) ( \beta + \delta_1    W_1^{(q)}(b-a)).
\end{align*}
\end{corollary}
\begin{corollary} \label{lemma_difference_at_a_b} 
 Fix $b > a\geq  0$ for the derivative at $b$ and $b > a > 0$ for the derivative at $a$.
(i) We have
\begin{align*} 
	\begin{split}
		g_{a,b}^{(q)\prime}(b+) -   g_{a,b}^{(q)\prime}(b-) 
		&=   \delta_1 \Phi_0(q)  e^{-\Phi_0(q) b} W_1^{(q)}(0), \\
		g_{a,b}^{(q)\prime}(a+) -   g_{a,b}^{(q)\prime}(a-) 
		&=   \delta_2 \Phi_0(q) W_2^{(q)}(0)  \Big[ e^{-\Phi_0(q) a} - \delta_1 l^{(q)\prime}(a;b)
		\Big].	
	\end{split}
\end{align*}
(ii) In particular, for the case of unbounded variation, we have
\begin{align*}
	g_{a,b}^{(q) \prime \prime}(b+)  - g_{a,b}^{(q) \prime \prime}(b-)  &= - \delta_1 \Phi_0(q)   e^{-\Phi_0(q) b} W_1^{(q) \prime} (0+), \\
	g_{a,b}^{(q) \prime \prime}(a+) - g_{a,b}^{(q) \prime \prime}(a-)  &=   - \delta_2 \Phi_0(q) W_2^{(q) \prime} (0+)  
	\Big[ e^{-\Phi_0(q) a}  - \delta_1 l^{(q)\prime}(a;b)
	\Big].
\end{align*}
\end{corollary}


\subsection{Smooth fit} 


We shall now obtain the condition on the smoothness of $v_{a,b}$ at $b$. By Corollaries \ref{lemma_difference_at_a_b_f}(i) and \ref{lemma_difference_at_a_b}(i) applied to Lemma \ref{value_fun_explicit},
\begin{align*}
v_{a,b}'(b+) - v_{a,b}'(b-) 
 = -\delta_1 W^{(q)}_1 (0) \frac {\Gamma(a,b)} {g^{(q)}_{a,b}(0)}, \quad 0 \leq a < b,
\end{align*}
where we define
\begin{align*}
\Gamma(a,b) := (f_{a,b}^{(q)}(0) - \rho) {\Phi_0(q)  e^{-\Phi_0(q) b}} - g_{a,b}^{(q)}(0), \quad 0 \leq a \leq b.
\end{align*}
Hence, for the case of bounded variation (with $W_1^{(q)}(0) > 0$ by \eqref{eq:Wqp0}), the smoothness at $b$ holds if 
\begin{align*}
\mathbf{C_b:} 
 \Gamma(a,b)  = 0.
\end{align*} 


On the other hand, for the case of unbounded variation,  by Corollaries \ref{lemma_difference_at_a_b_f}(ii) and \ref{lemma_difference_at_a_b}(ii) applied to Lemma \ref{value_fun_explicit},

\begin{align*}
v_{a,b}''(b+) - v_{a,b}''(b-) 
 = \delta_1 W^{(q)\prime}_1 (0+) \frac {\Gamma(a,b)} {g^{(q)}_{a,b}(0)}, \quad 0 \leq a < b.
\end{align*}
Hence, $v_{a,b}$ is twice continuously differentiable at $b$ if $\mathbf{C_b}$ holds.

Similarly, we shall obtain the condition on the smoothness of $v_{a,b}$ at $a$.
We define, for $0 <a < b$,
 \begin{align}\label{def_gamma}
 \begin{split}
 	\gamma(a,b) :=\frac \partial {\partial a} \Gamma(a,b) &= \Phi_0(q)   \delta_2  W_2^{(q)} (a) \Big[ e^{-\Phi_0(q) b} ( \delta_1 W_1^{(q)}(b-a) + \beta )-     e^{-\Phi_0(q) a} +  \delta_1  l^{(q)\prime}(a;b) \Big]
	\\&= \Phi_0(q)   \delta_2  W_2^{(q)} (a) \Big[ e^{-\Phi_0(q) b}  \beta -    \Big( e^{-\Phi_0(q) a} + \delta_1  \Phi_0(q) l^{(q)}(a;b) 
	\Big) \Big].
	\end{split}
 \end{align}
 
For $0 < a < b$, Corollaries \ref{lemma_difference_at_a_b_f}, and \ref{lemma_difference_at_a_b},
\begin{align*}
v'_{a,b}(a+) - v_{a,b}'(a-) 
 &= \delta_2 W^{(q)}_2 (0)  \Big[ \beta + \delta_1 W_1^{(q)}(b-a)  - (f_{a,b}^{(q)}(0) - \rho) \frac { \Phi_0(q)  \big( e^{-\Phi_0(q) a} - \delta_1 l^{(q)\prime}(a;b)
 \big) } {g_{a,b}^{(q)}(0)} \Big] \\
 &=   \frac {W_2^{(q)}(0)} {g^{(q)}_{a,b}(0)}  \Big[(f_{a,b}^{(q)}(0) - \rho)   \frac {\gamma(a,b)}  {W_2^{(q)}(a)} -  \delta_2 (\beta + \delta_1 W_1^{(q)}(b-a)) \Gamma(a,b) \Big],
\end{align*}
and, for the unbounded variation case,
	\begin{align*}
		v_{a,b}''(a+) - v_{a,b}''(a-) 
		 &=  - \frac {W_2^{(q)\prime}(0+)} {g^{(q)}_{a,b}(0)}  \Big[(f_{a,b}^{(q)}(0) - \rho)   \frac {\gamma(a,b)}  {W_2^{(q)}(a)} -  \delta_2 (\beta + \delta_1 W_1^{(q)}(b-a)) \Gamma(a,b) \Big].
\end{align*}

Therefore, if condition $\mathbf{C_b}$ is satisfied and $a>0$, the continuous differentiability (resp.\ twice continuous differentiability) at $a$ for the case of bounded (resp.\ unbounded variation) holds on condition that
 \begin{align*}
 	\mathbf{C_a:}      \frac \partial {\partial a} \Gamma(a,b) = \gamma(a,b)= 0.
 \end{align*}
We shall now summarize the results below.
\begin{lemma} \label{smooth_fit_prob1}
(1) Suppose $\mathbf{C_b}$ is satisfied for $0 \leq a < b$. 
Then, $v_{a,b}$ is continuously differentiable (resp.\ twice continuously differentiable) at $b$ for the case of bounded (resp.\ unbounded) variation.

(2) Suppose in addition that $\mathbf{C}_a$ is satisfied for $0 < a < b$. Then, $v_{a,b}$ is continuously differentiable (resp.\ twice continuously differentiable) at $a$ for the case of bounded (resp.\ unbounded) variation.

\end{lemma}


								
\subsection{ The existence of $(a^*, b^*)$} \label{subsection_existence_a_b}
%

Let
\begin{align*}
	\tilde{\gamma}(a,b):= &\frac {\gamma(a,b)} {W_2^{(q)}(a)} = \Phi_0(q)   \delta_2   \Big( e^{-\Phi_0(q) b}  \beta - 
	\Big[ e^{-\Phi_0(q) a} + \delta_1  \Phi_0(q) l^{(q)}(a;b) 
	\Big] \Big), \quad 0 < a \leq b.
\end{align*}
Then we have, for $0<a<b$, 	by \eqref{l_prime},
\begin{align*}
	& \frac \partial {\partial a} \tilde{\gamma}(a,b)
	 = \Phi_0(q)   \delta_2 
	\Big[ \Phi_0(q) e^{-\Phi_0(q) a} - \delta_1 \Phi_0(q) l^{(q)\prime}(a;b)
	 \Big]>0.
\end{align*}	
In addition, for $a=b$, we get
\begin{align*}
	\tilde{\gamma}(b,b) 
	= \Phi_0(q)   \delta_2    e^{-\Phi_0(q) b}  (\beta - 1) > 0.\end{align*}


For $b \geq 0$, in view of the above computations and the positivity of $W_2^{(q)}(a)$, there exists $a(b) \geq 0$ such that
$\gamma(a,b) < 0$ for $a < a(b)$ and 
$\gamma(a,b) > 0$ for  $a > a(b)$.
Hence, 
\begin{enumerate}
\item when $a(b) > 0$ (or equivalently $\lim_{a \downarrow 0} \tilde{\gamma}(a,b)
< 0$), then  $a \mapsto \Gamma(a,b)$ is decreasing on $(0, a(b))$ and increasing on $(a(b), b)$;
\item  when $a(b) = 0$ (or equivalently $\lim_{a \downarrow 0} \tilde{\gamma}(a,b)
\geq 0$), then $a \mapsto \Gamma(a,b)$ is increasing.
\end{enumerate}
Therefore, for each $b  \geq 0$,  $a(b)$ is the minimizer such that
\begin{align*}
\underline{\Gamma}(b) := \Gamma(a(b),b)  = \min_{0 \leq a \leq b} \Gamma(a,b).
\end{align*}

Because $f^{(q)}_{0,0}(0)= \delta_1 /q$ and $g_{0,0}^{(q)}(0)=1$, we have
 \begin{align*}
 \underline{\Gamma}(0)=\Gamma(0,0)=\left(\frac{\delta_1}{q}-\rho\right)\Phi_0(q)-1.
 \end{align*}
 Hence we have that $\underline{\Gamma}(0)\leq 0$ if and only if
 \begin{equation}
 \label{cond_degenerate}
 \left(\frac{\delta_1}{q}-\rho\right)\Phi_0(q)\leq 1.
 \end{equation}

\begin{lemma}\label{lemma_Gamma_bar_monotone} 
\label{lemma_gamma_decreasing} If $ \delta_1 /q>\rho$, then  $\underline{\Gamma}$ is strictly decreasing to $-\infty$.
\end{lemma}
\begin{proof}
For $0 \leq a < b$, because
\begin{align*}
\frac \partial {\partial b} f_{a,b}^{(q)}(0) &= \delta_1 \Big( W_1^{(q)} (b)+ \delta_2  \int^{a}_{0} W_2^{(q)} (u)  W^{(q)\prime}_1 (b-u)  \diff u \Big), \\
\frac \partial {\partial b}g_{a,b}^{(q)}(0) 
 &=  \delta_1 \Phi_0(q) e^{-\Phi_0(q) b} \Big( W_1^{(q)} (b)  +	  \delta_2  \int^{a}_{0} W_2^{(q)}(u) W_1^{(q)\prime} (b-u)  \diff u \Big) 
=\Phi_0(q) e^{-\Phi_0(q) b} \frac \partial {\partial b} f_{a,b}^{(q)}(0) ,
	 \end{align*}
	we have 
\begin{align*}
\frac \partial {\partial b}\Gamma(a,b) & = 
  - (f_{a,b}^{(q)}(0)-\rho) {\Phi_0^2(q)  e^{-\Phi_0(q) b}}  \\
 &=   - \Phi_0^2(q)  e^{-\Phi_0(q) b}  \Big[ \delta_1 \Big( \overline{W}_1^{(q)} (b)+ \delta_2  \int^{a}_{0} W_2^{(q)} (u)  W^{(q)}_1 (b-u)  \diff u \Big)  +\beta\delta_2 \overline{W}^{(q)}_2 (a) + \frac {\delta_1} q -\rho \Big], \end{align*}
 which admits a negative upper bound by the assumption $\delta_1 /q>\rho$:
 	 \begin{align*}
 	 	\sup_{0 \leq a < b}\frac \partial {\partial b}\Gamma(a,b)  &=   - \Phi_0^2(q)  e^{-\Phi_0(q) b} \left( \delta_1\overline{W}_1^{(q)} (b)+ \frac {\delta_1} q-\rho\right) < 0. \end{align*}
 	 On the other hand, using Exercise 8.5 (i) in \cite{K}, identity \eqref{W_q_limit}, and the fact that $\Phi_1(q)> \Phi_0(q)$ for $q>0$, we have that $\lim_{b\to\infty}e^{-\Phi_0(q) b}\overline{W}_1^{(q)} (b)=\infty$.
 	 Hence, we have $\sup_{b > 0}\sup_{0 \leq a < b}\frac \partial {\partial b}\Gamma(a,b) < 0$.
This shows the claim.
 \end{proof}
 
In order to obtain the required smoothness for the function $v_{a^*,b^*}$, we shall choose $(a^*, b^*)$ as follows. Suppose $ \delta_1 / q >\rho$. 
\begin{enumerate}
 \item If $\underline{\Gamma}(0) > 0$, then we increase the value of $b$ until we attain $b^*$ such that $\underline{\Gamma}(b^*) = 0$ 
 (which exists and is unique by Lemma  \ref{lemma_gamma_decreasing}).  We set $a^* = a(b^*)$, which is either zero or positive. 
 By construction, $\mathbf{C_b}$ is satisfied. When $a^* > 0$, $a \mapsto \Gamma(a,b^*)$ attains a local minimum at $a^*$ and hence 
 $\mathbf{C}_a$ is satisfied. 
 When $a^* = 0$, 
 we must have
 \begin{align}
 	\gamma(0, b^*) \geq 0, \label{gamma_zero_b}
 \end{align}
 because otherwise, by the relation \ref{def_gamma}, the function $a \mapsto \Gamma(a, b^*)$ must attain a negative value (which contradicts with $\underline{\Gamma}(b^*) = 0$).
 \item If $\underline{\Gamma}(0)=\Gamma(0,0) \leq 0$,  we set $a^* =b^*=0$.
\end{enumerate}
Suppose $\delta_1 / q \leq \rho$. In this case, the terminal payoff $\rho$ dominates the perpetual dividend payoff $\int_0^\infty e^{-qt} \delta_1 \diff t$ and it is reasonably conjectured that it is optimal to liquidate as quickly possible. 
Hence, we set $a^* = b^* = 0$.

 
\section{Optimality} \label{section_optimality}
In this section we will prove the optimality of the proposed multi-refraction strategy.

Suppose $b^* > 0$.  Because in this case condition $\mathbf{C_b}$ is satisfied, in view of Lemma \ref{value_fun_explicit} and using \eqref{l_prime}, we have, for $x \geq 0$,
\begin{align} \label{form_value_kazu}
\begin{split}
	v_{a^*,b^*}(x)
					&= f_{a^*,b^*}^{(q)}(x) - \frac{e^{\Phi_0(q) b^*}}{\Phi_0(q)} g_{a^*,b^*}^{(q)}(x) \\
					&= \delta_1  \overline{W}_1^{(q)} (b^*-x)  +\beta\delta_2 \overline{W}^{(q)}_2 (a^*-x) + \frac {\delta_1} q - \frac{e^{-\Phi_0(q)(x-b^*)}}{\Phi_0(q)} - \delta_1  e^{\Phi_0(q)b^*}
					l^{(q)}(x; b^*)
					\\
	&-	\delta_2  e^{\Phi_0(q)b^*}   \int^{a^*}_{x} W_2^{(q)}(z-x)  \Big[ e^{-\Phi_0(q) z} + \delta_1   \Phi_0(q) 
	l^{(q)}(z; b^*)
	\Big]  \diff z.
	\end{split}
					 \end{align}

On the other hand when $b^*=0$, Lemma \ref{value_fun_explicit} gives 
		\begin{equation}\label{form_value_0}
		v_{0,0}(x)=\frac{\delta_1}{q}-\left(\frac{\delta_1}{q}-\rho\right)e^{-\Phi_0(q)x} \qquad\textbf{$x\geq 0$}.
		\end{equation}
		
		\begin{theorem}\label{optimalv}
The multi-refraction strategy $\pi^{a^*,b^*}$ is optimal and the value function of the problem \eqref{v_pi_1} is given by $v(x)=v_{a^*,b^*}(x)$ for all $x\geq0$.
\end{theorem}


We call a function \emph{sufficiently smooth} on $(0, \infty)$ if it is differentiable (resp.\ twice differentiable) on $(0, \infty)$ when $-X_1 = Y$ is of bounded (resp.\ unbounded) variation.

We let $\mathcal{L}_{-X_1}$ be the operator acting on sufficiently smooth functions $g$, defined by
\begin{equation} \label{generator_Y}
\begin{split}
\mathcal{L}_{-X_1} g(x)&:= - \gamma_1 g'(x)+\frac{\sigma^2}{2}g''(x) +\int_{(0,\infty)}[g(x + z)-g(x)-g'(x)z\mathbf{1}_{\{0<z<1\}}]\Pi(\mathrm{d}z).
\end{split}
\end{equation}
The proof of the following lemma is deferred to Appendix \ref{proof_verificationlemma}.
\begin{lemma}[Verification lemma]
	\label{verificationlemma}
	Suppose $\hat{\pi}$ is an admissible dividend strategy such that $v_{\hat{\pi}}$ is sufficiently smooth on $(0,\infty)$ 
	and satisfies

			\begin{align} \label{HJB-inequality}
		(\mathcal{L}_{-X_1} - q)v_{\hat{\pi}}(x)+ \sup_{0\leq r\leq\delta_1} r (1-v'_{\hat{\pi}}(x) \big)+ \sup_{0\leq r\leq\delta_2}  r (v'_{\hat{\pi}}(x)-\beta  \big) &\leq 0, \quad x   > 0, \\
				v_{\hat{\pi}}(0+)&=\rho. \label{v_at_zero}
	\end{align} 
	Then $v_{\hat{\pi}}(x)=v(x)$ for all $x\in\R$ and hence $\hat{\pi}$ is an optimal strategy.
\end{lemma}

\begin{proposition}\label{convexity}
	(i) Suppose $\delta_1 / q >\rho$. The function $v_{a^*,b^*}$ is concave on $(0,\infty)$, $v_{a^*,b^*}'(a^*)\leq \beta$ (with equality when $a^* > 0$), and $v_{a^*,b^*}'(b^{*})\leq1$ (with equality when $b^* > 0$).
	
	(ii) Suppose $\delta_1 / q \leq\rho$. The function $v_{a^*,b^*}$ is convex on $(0,\infty)$, and $v_{a^*,b^*}'(x) \leq 0$ for all $x>0$.
\end{proposition}
\begin{proof}
(i)  
(1) 
	Let us suppose $b^*>0$.  For $x > a^*$, differentiating \eqref{form_value_kazu} and by \eqref{l_prime}, 
\begin{align}
\label{vf_der_def}
\begin{split}
	v_{a^*,b^*}'(x) 
					&=  e^{-\Phi_0(q) (x-b^*)} + \delta_1 \Phi_0(q) e^{\Phi_0(q)b^*} l^{(q)}(x; b^*)  
					\end{split}
					\end{align}
and,	by \eqref{l_prime_double_prime},
\begin{align*}
	v_{a^*,b^*}''(x)
					&=    - \Phi_0(q) e^{-\Phi_0(q)(x-b^*)}   + \delta_1  \Phi_0(q)  e^{\Phi_0(q)b^*} l^{(q)\prime}(x;b^*).
					 \end{align*}
					 Hence, by the nonpositivity of $l^{(q)\prime}$ as in \eqref{l_prime}, $v_{a^*, b^*}$ is concave on $(a^*, \infty)$. 
Now suppose $a^* > 0$ (which means $\gamma(a^*, b^*) = 0$) and consider the case $0 < x < a^*$. 
We have, using Lemmas \ref{lemma_derivatives_f} and \ref{lemma_derivatives_g}, 
 that 
\begin{align}\label{v_sec_der_aux}
\begin{split}
	v_{a^*,b^*}''(x-) &=\delta_1 \delta_2  \Big( W_1^{(q)}(b^*-a^*) W_2^{(q)\prime}((a^*-x) +)+ W_2^{(q)} (0)  W^{(q) \prime}_1 ((b^*-x)+)
	\\&+  \int^{a^*}_{x} W_2^{(q)\prime} (u-x)  W^{(q) \prime}_1 (b^*-u)  \diff u  \Big) 
	+ \delta_1  W_1^{(q) \prime} ((b^*-x)+) + \beta\delta_2 W^{(q)\prime}_2 ((a^*-x)+)  \\
	&- e^{\Phi_0(q) b^*} \Big[ \Big( \Phi_0(q) e^{-\Phi_0(q) x} + \delta_1  l^{(q)\prime \prime}(x-;b^*) \Big) \Big( 1 + \delta_2 W_2^{(q)} (0) 
	\Big) \\
	&\qquad \qquad +  \delta_2 \Big(  W^{(q)\prime}_2((a^*-x)+) \Big( e^{- \Phi_0(q)a^*} - \delta_1 l^{(q)\prime}(a^*,b^*) \Big)  \\ &\qquad \qquad + \int^{a^*}_{x} W_2^{(q)\prime}(z-x)  \Big( \Phi_0(q) e^{-\Phi_0(q) z} + 
	\delta_1 l^{(q)\prime \prime}(z;b^*)  \Big) \diff z \Big) \Big]. 
	\end{split}
\end{align}

Note that \eqref{l_prime}, \eqref{def_gamma} and the fact that $\gamma(a^*, b^*) = 0$ imply 
\begin{align*}
	e^{- \Phi_0(q)a^*} - \delta_1 l^{(q)\prime}(a^*; b^*) =   \delta_1  e^{- \Phi_0(q) b^*} W_1^{(q)}(b^*-a^*) + e^{-\Phi_0(q)b^*} \beta.
\end{align*}
Substituting the above expression and \eqref{l_prime_double_prime} in \eqref{v_sec_der_aux}, we have
\begin{align*}
	& v_{a^*,b^*}''(x-)  \\ &=\delta_1 \delta_2  \Big( W_1^{(q)}(b^*-a^*) W_2^{(q)\prime}((a^*-x)+)+ W_2^{(q)} (0)  W^{(q) \prime}_1 ((b^*-x)+)
	\\ &\qquad \qquad +  \int^{a^*}_{x} W_2^{(q)\prime} (u-x)  W^{(q) \prime}_1 (b^*-u)  \diff u  \Big) 
	  \\ &+ \delta_1  W_1^{(q) \prime} ((b^*-x)+) + \beta\delta_2 W^{(q)\prime}_2 ((a^*-x)+) \\
	&- e^{\Phi_0(q) b^*} \Big[ \Big\{ \Phi_0(q) e^{-\Phi_0(q) x} + \delta_1 \Big(  e^{-\Phi_0(q)b^*} W^{(q)\prime}_1((b^*-x)+) - \Phi_0(q)  l^{(q)\prime}(x;b^*) \Big)  \Big\} \Big( 1 + \delta_2 W_2^{(q)} (0)
	 \Big) \\
	&\qquad \qquad +  \delta_2  \Big\{  W^{(q)\prime}_2((a^*-x)+) \Big(\delta_1  e^{- \Phi_0(q) b^*} W_1^{(q)}(b^*-a^*) + e^{-\Phi_0(q)b^*} \beta \Big)  \\ &\qquad \qquad + \int^{a}_{x} W_2^{(q)\prime}(z-x)  \Big( \Phi_0(q) e^{-\Phi_0(q) z} + \delta_1 (  e^{-\Phi_0(q)b^*} W^{(q)\prime}_1(b^*-z) - \Phi_0(q)  l^{(q)\prime}(z;b^*) ) \Big) \diff z \Big\} \Big] \\
	&=
	- e^{\Phi_0(q) b^*} \Big[  \Big( \Phi_0(q) e^{-\Phi_0(q) x} - \delta_1  \Phi_0(q)  l^{(q)\prime}(x;b^*) \Big) \Big( 1 + \delta_2 W_2^{(q)} (0) 
	\Big) \\
	&+  \delta_2  \Big(    \int^{a^*}_{x} W_2^{(q)\prime}(z-x) \Phi_0(q)  \Big(  e^{-\Phi_0(q) z} - \delta_1   l^{(q)\prime}(z;b^*) \Big) \diff z \Big) \Big].
\end{align*}

This is negative  for all $x>0$ by the nonpositivity  of $l^{(q)\prime}$ as in \eqref{l_prime}.
\par Now, from \eqref{vf_der_def}, we deduce that $v'_{a^*,b^*}(b^*)=1$ and
	\begin{align*}
	v'_{a^*,b^*}(a^* +)&=  e^{-\Phi_0(q) (a^*-b^*)} + e^{\Phi_0(q)b^*}\delta_1 \Phi_0(q) l^{(q)}(a^*; b^*)  
	\end{align*}
which
equals $\beta$ when $a^* > 0$ because condition $\mathbf{C}_a$ is satisfied, and is less than
  or equal to $\beta$ when $a^* = 0$, because of \eqref{gamma_zero_b}.

(2) Let us now suppose $b^*=0$. Differentiating \eqref{form_value_0}, we obtain that 
\begin{align}\label{der_v_00}
\begin{split}
	v_{0,0}'(x)&= \left(\frac{\delta_1}{q}-\rho\right)\Phi_0(q)e^{-\Phi_0(q)x}>0,\\
	v_{0,0}''(x)&=- \left(\frac{\delta_1}{q}-\rho\right) \Phi_0^2(q)e^{-\Phi_0(q)x}< 0.
	\end{split}
\end{align}
This implies that the function is strictly concave and hence 
\[
v_{0,0}'(x)< v_{0,0}'(0+)=\left(\frac{\delta_1}{q}-\rho\right)\Phi_0(q)\leq 1, \quad x > 0,
\]
where the last inequality follows from \eqref{cond_degenerate}.
\par (ii) Finally, let us consider the case $\delta_1 / q \leq \rho$. Because the equalities of  
 \eqref{der_v_00} hold in this case, for $x>0$, $v_{0,0}'(x)\leq0$, and $v_{0,0}''(x) \geq0$. This shows the result.
\end{proof}

		
		The next result follows by a direct application of Proposition \ref{convexity}.
		\begin{lemma}\label{cond_max_v}
We have
\begin{align*}
\sup_{0\leq r\leq \delta_1}r (1-v_{a^*,b^*}'(x))+ \sup_{0\leq r\leq\delta_2} r (v'_{a^*,b^*}(x)-\beta ) =
\begin{cases}
	  \delta_2 (v_{a^*,b^*}'(x)-\beta) &\text{if $0<x< a^*$}, \\
	0 & \text{if $a^*\leq x< b^*$}, \\
	\delta_1 (1 -v_{a^*,b^*}'(x) ) & \text{if $x\geq b^*$}.
\end{cases}
\end{align*}
\end{lemma}



		\begin{lemma}\label{generator_on_v}
		Recall the generator $\mathcal{L}_{-X_1}$ as in \eqref{generator_Y}.
		We have 
		\begin{align*}
		(\mathcal{L}_{-X_1} - q) v_{a^*,b^*}(x)=
		\begin{cases}
			-\delta_2 (v_{a^*,b^*}'(x)-\beta)&\text{if $0<x< a^*$}, \\
			0 & \text{if $a^*\leq x< b^*$}, \\
			\delta_1(v_{a^*,b^*}'(x)- 1) & \text{if $x\geq b^*$}.
		\end{cases}
	\end{align*}
		\end{lemma}
		\begin{proof}
	

		Recall the relationship \eqref{X_k_def}. 
		For $k = 0,2$  and any sufficiently smooth function $g$, let $\mathcal{L}_{-X_k}$ be the operator corresponding to the process $X_k$, given by
		\begin{align} \label{generator_relation}
		\mathcal{L}_{-X_0} g = \mathcal{L}_{-X_1} g -\delta_1 g', \quad 
		\textrm{and} \quad \mathcal{L}_{-X_2}g = \mathcal{L}_{-X_1}g +\delta_2 g'.
		\end{align}

First as in, for instance the proof of Theorem 2.1 in \cite{BKY},
\begin{align} \label{generator_Z_harmonic}
(\mathcal{L}_{-X_k}-q)Z_k^{(q)}(x)=0, \quad k=1,2, \quad x > 0.
\end{align}
In addition, direct computation gives
		\begin{align}
		(\mathcal{L}_{-X_0}-q)e^{-\Phi_0(q) x}=0. \label{generator_X_0_exp}
		\end{align}
		On the other hand, by the proof of Lemma 4.5 of \cite{BKY2}, we have that, for $x < b^*$,
		\begin{align} \label{generator_running}
		\begin{split}
		&(\mathcal{L}_{-X_k}-q) 
		\int^{b^*}_{x} e^{-\Phi_0(q) z} W_{k}^{(q)} (z-x) \diff z
		=e^{-\Phi_0(q) x}\qquad\text{for $k=1,2$,}\\
		&(\mathcal{L}_{-X_2}-q)\int^{b^*}_{x} W_2^{(q)} (z-x)
		l^{(q)}(z; b^*)
		\diff z=
		l^{(q)}(x; b^*).
		\end{split}
		\end{align}
		Now we write \eqref{form_value_kazu} as, for $x \geq 0$,
\begin{align} \label{value_function_version2}
\begin{split}
	v_{a^*,b^*}(x)
					&= \delta_1  \frac {Z_1^{(q)} (b^*-x)} q  +\beta\delta_2 \frac {Z^{(q)}_2 (a^*-x) -1} q  - \frac{e^{-\Phi_0(q)(x-b^*)}}{\Phi_0(q)}  - e^{\Phi_0(q)b^*}\delta_1  l^{(q)}(x; b^*) 
					 \\
	&-	  e^{\Phi_0(q)b^*} \delta_2  \int^{a^*}_{x} W_2^{(q)}(z-x)  \Big[ e^{-\Phi_0(q) z} + \delta_1   \Phi_0(q) 
	l^{(q)}(z; b^*)
	\Big]  \diff z.
	\end{split}
					 \end{align}
		
		(i) Suppose $b^*>0$. For $x > b^*$, by \eqref{generator_X_0_exp}, we have
						\begin{align*}
				(\mathcal{L}_{-X_0} - q) v_{a^*,b^*}(x)  &=  (\mathcal{L}_{-X_0} - q) \frac {\delta_1} q -    (\mathcal{L}_{-X_0} - q) \frac {e^{-\Phi_0(q)(x-b^*)}} {\Phi_0(q)} = - \delta_1.
			\end{align*}
For $a^* < x < b^*$, we have, by 
equations \eqref{generator_relation}-\eqref{generator_running}
applied to \eqref{value_function_version2},
			\begin{align*}
	(\mathcal{L}_{-X_1} - q) v_{a^*,b^*}(x)
					&= (\mathcal{L}_{-X_1} - q) \Big( \delta_1  \frac {Z_1^{(q)} (b^*-x) } q  - \frac{e^{\Phi_0(q)b^*}}{\Phi_0(q)}  e^{-\Phi_0(q) x} - e^{\Phi_0(q)b^*}\delta_1 l^{(q)}(x; b^*)
					 \Big)  \\
								&=- \delta_1 \Big(  - e^{\Phi_0(q)b^*}  e^{-\Phi_0(q) x} \Big)  - \delta_1  e^{-\Phi_0(q) (x-b^*)} =0.
			 \end{align*}
			 
					 Suppose $x < a^*$. By 
						equations \eqref{generator_relation}-\eqref{generator_X_0_exp},
					 \begin{align*}
					(\mathcal{L}_{-X_2} - q) Z_1^{(q)} (b^*-x) &=  (\mathcal{L}_{-X_1}-q) Z_1^{(q)} (b^*-x)  +\delta_2 \frac \partial {\partial x} Z_1^{(q)} (b^*-x)= - \delta_2 q W_1^{(q)} (b^*-x), \\
					(\mathcal{L}_{-X_2} - q)  e^{-\Phi_0(q) x}  
					&= (\mathcal{L}_{-X_0}-q)  e^{-\Phi_0(q) x}    +(\delta_1 + \delta_2) \frac \partial {\partial x}  e^{-\Phi_0(q) x}  =   - (\delta_1 + \delta_2) \Phi_0(q) e^{-\Phi_0(q)x}.
										\end{align*}
				On the other hand, by \eqref{generator_relation} and  \eqref{generator_running}, together with \eqref{l_prime}, 
					\begin{multline*}
					(\mathcal{L}_{-X_2} - q) l^{(q)}(x; b^*) 
					=   e^{-\Phi_0(q) x}  + \delta_2 l^{(q)\prime}(x; b^*) 
					=  e^{-\Phi_0(q) x}  - \delta_2 \Big( e^{-\Phi_0(q) b^*} W_1^{(q)} (b^*-x) 
					+ \Phi_0(q) l^{(q)}(x; b^*)
					 \Big)
					 \end{multline*}
and, by \eqref{generator_running},
\begin{align*}
(\mathcal{L}_{-X_2} - q)   \int^{a^*}_{x} W_2^{(q)}(z-x)  \Big[ e^{-\Phi_0(q) z} + \delta_1   \Phi_0(q) l(z; b^*)
\Big]  \diff z = e^{-\Phi_0(q) x} + \delta_1   \Phi_0(q) l^{(q)}(x; b^*). 
 \end{align*}
Substituting these in \eqref{value_function_version2}, we have $(\mathcal{L}_{-X_2} - q) v_{a^*,b^*}(x) = \beta \delta_2$. Hence the result holds by \eqref{generator_relation}.
					 
(ii) Now let us assume that $b^*=0$.
					For $x > 0$, we have by \eqref{form_value_0} and \eqref{generator_X_0_exp},
					\begin{align*}
						(\mathcal{L}_{-X_0} - q) v_{0,0}(x)  &=  (\mathcal{L}_{-X_0} - q) \frac {\delta_1} q -   \left(\frac {\delta_1} q-\rho\right) (\mathcal{L}_{-X_0} - q) e^{-\Phi_0(q)x}  = - \delta_1.
					\end{align*}
					This together with \eqref{generator_relation} completes the proof.
					 \end{proof}

Now we have all the elements to prove Theorem \ref{optimalv}. 
\begin{proof}[Proof of Theorem \ref{optimalv}]
	 
	 First, by Lemma \ref{smooth_fit_prob1} and how the thresholds $a^*$ and $b^*$ are chosen, $v_{a^*, b^*}$ is sufficiently smooth.  In addition, Lemmas  \ref{cond_max_v} and  \ref{generator_on_v} show \eqref{HJB-inequality}, and Lemma \ref{value_fun_explicit} shows  \eqref{v_at_zero}. Hence, by Lemma \ref{verificationlemma}, the proof is complete. 
	 
\end{proof}

\section{Numerical Examples} \label{numerical_section}

We conclude the paper with numerical examples of the optimal dividend problem studied above. Here, 
we focus on the case 
%
\begin{equation}
 Y(t)  - Y(0)= - c_Y t+\sigma B(t) + \sum_{n=1}^{N(t)} Z_n, \quad 0\le t <\infty, \label{X_phase_type}
\end{equation}
for some $c_Y \in \R$ 
and $\sigma > 0$.  Here $B=( B(t); t\ge 0)$ is a standard Brownian motion, $N=(N(t); t\ge 0 )$ is a Poisson process with arrival rate $\kappa$, and  $Z = ( Z_n; n = 1,2,\ldots )$ is an i.i.d.\ sequence of phase-type-distributed random variables with representation $(m,{\bm \alpha},{\bm T})$; see \cite{Asmussen_2004}.
These processes are assumed mutually independent. 

The Laplace exponents of $-Y=X_1$ and $X_k$ ($k=0,2$) 
are then (with ${\bm t} = -\bm{T 1}$ where ${\bm 1} = [1, \ldots 1]'$)
\begin{align*}
 \psi_0(\lambda)   &= (c_Y+ \delta_1 )\lambda + \frac 1 2 \sigma^2 \lambda^2 + \kappa \left( {\bm \alpha} (\lambda {\bm I} - {\bm{T}})^{-1} {\bm t} -1 \right), \\
 \psi_k(\lambda) &= \psi_0(\lambda) -\sum_{i=1}^k \delta_i \lambda,  \quad  \textrm{$k=1,2,$}
\end{align*}
which are analytic for every $s \in \mathbb{C}$ except at the eigenvalues of ${\bm T}$.  

Suppose  $( -\xi_{i,q}^{(k)}; i \in \mathcal{I}_q )$ are the sets of the roots with negative real parts of the equality $\psi_k(s) = q$.  We assume that the phase-type distribution is minimally represented and hence $|\mathcal{I}_q| = m+1$ as $\sigma > 0$; see \cite{Asmussen_2004}.  As in \cite{Egami_Yamazaki_2010_2}, if these values are assumed distinct, then
the scale functions of $X_k$ for $k=0,1,2$ can be written, for all $x \geq 0$, 
\begin{align}
W^{(q)}_k(x) = \frac {e^{\Phi_k(q) x}} {\psi_k'(\Phi_k(q))} + \sum_{i \in \mathcal{I}_q} \frac 1 {\psi_k'(-\xi^{(k)}_{i,q})} e^{-\xi^{(k)}_{i,q}x},
\label{scale_function_form_phase_type}
\end{align}
respectively.
Let $q = 0.05$, $r = 0.05$, $\delta_1 = 1$, $\delta_2 = 0.5$, and $\beta = 1.2$.
For the process $Y$ in \eqref{X_phase_type}, we set $\sigma = 0.2$, $c_Y = 0.5$, and $\kappa = 1$ and, for $Z$, we use the phase-type distribution with  $m=6$ 
that gives an approximation to the (folded) normal random variable with mean $0$ and variance $1$ (see \cite{leung2015analytic} for the values of $\alpha$ and ${\bm T}$).  For the terminal payoff $\rho$, we consider \textbf{Case 1} with $\rho = 0$  and \textbf{Case 2} with $\rho = 5$ to obtain, respectively, the cases $a^* > 0$ and $a^* = 0$ (and $b^* > 0$).

\subsection{Computation of $(a^*, b^*)$ and  $v_{a^*,b^*}$}

The first implementation step is the computation of the optimal thresholds $(a^*, b^*)$.  
As we discussed in Section \ref{subsection_existence_a_b}, we choose these such that $\underline{\Gamma}(b^*) = 0$ and 
$a^* = \arg \min_{a \geq 0} \Gamma(a,b^*)$  if $\underline{\Gamma}(0) > 0$ (equivalently, \eqref{cond_degenerate} does not hold); 
for $\underline{\Gamma}(0) \leq 0$, we set $a^* = b^* = 0$.  In both \textbf{Cases 1}  and \textbf{2}, we have $\underline{\Gamma} (0) > 0$ and hence $b^* > 0$.
As we have shown in Section \ref{subsection_existence_a_b}, $a \mapsto \Gamma(a,b)$ attains its minimum at $a(b) \geq 0$, which can be computed by bisection applied to the function $\gamma(\cdot, b)$. 
In addition, by Lemma \ref{lemma_Gamma_bar_monotone}, we can obtain the root of $\underline{\Gamma} (\cdot) = 0$ by another bisection method; the root becomes $b^*$ and  $a^* = a(b^*)$.  With these $a^*$ and $b^*$, the value function becomes $v_{a^*, b^*}$ as in \eqref{form_value_kazu} when $b^* > 0$ and \eqref{form_value_0} when $b^* = 0$.

In Figure \ref{plot_gamma}, we plot, for \textbf{Cases 1} and \textbf{2}, the mappings $a \mapsto \Gamma(a,b)$ and $a \mapsto \gamma(a,b)$ 
for various values of $b$.  As in the proof of Lemma \ref{lemma_Gamma_bar_monotone}, $\Gamma(a,b)$ is decreasing in $b$. 
The value $a(b)$ corresponds to the point (indicated by down-pointing triangles) at which $\Gamma(\cdot, b)$ attains a minimum and $\gamma(\cdot, b)$ vanishes (if $a(b) > 0$). 
The curve $a \mapsto \Gamma(a,b^*)$ touches and gets tangent to the x-axis at $a^*$.
\begin{figure}[h!]
\begin{center}
\begin{minipage}{0.9\textwidth}
\centering
\begin{tabular}{cc}
 \includegraphics[scale=0.33]{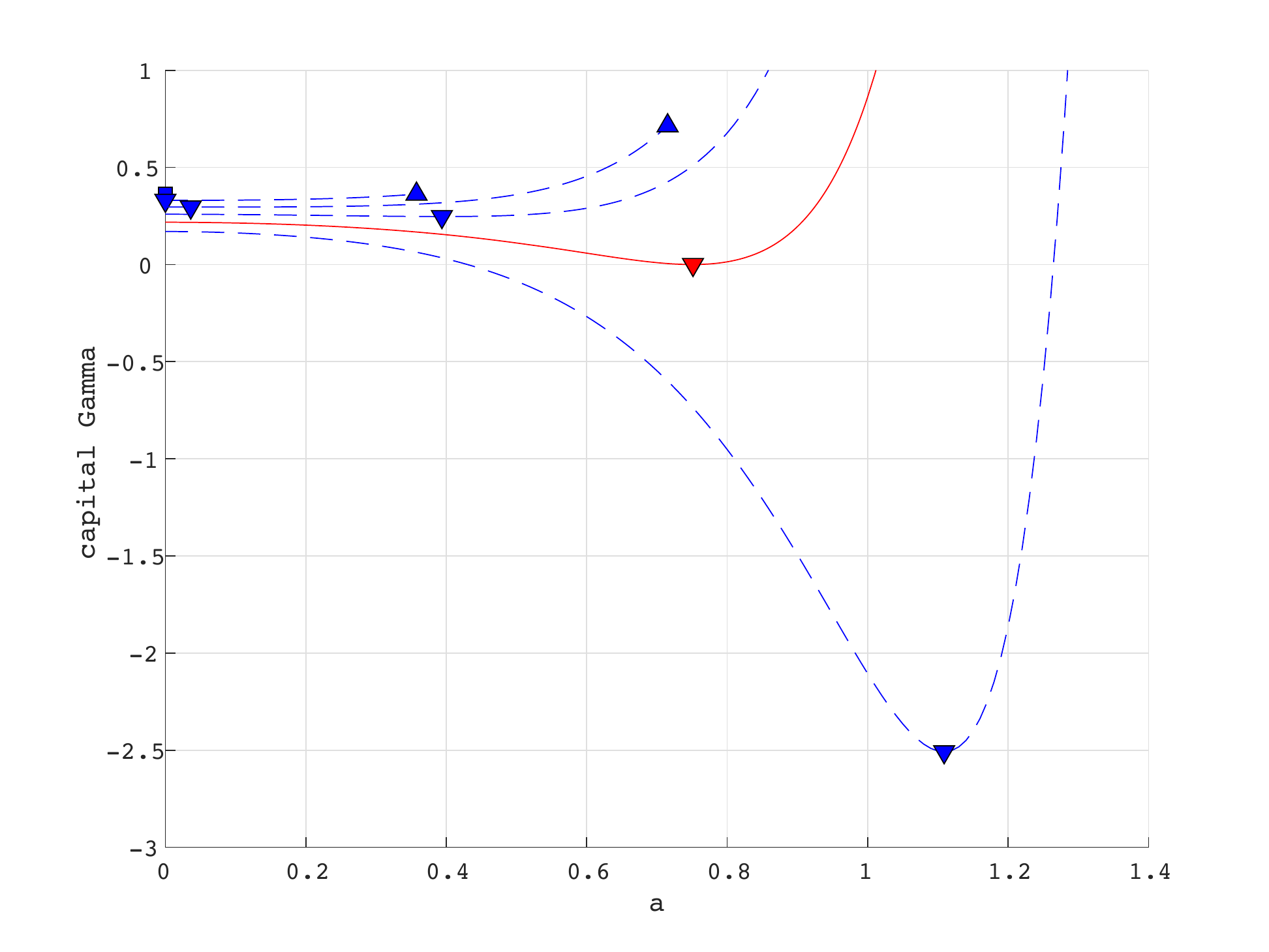} & \includegraphics[scale=0.33]{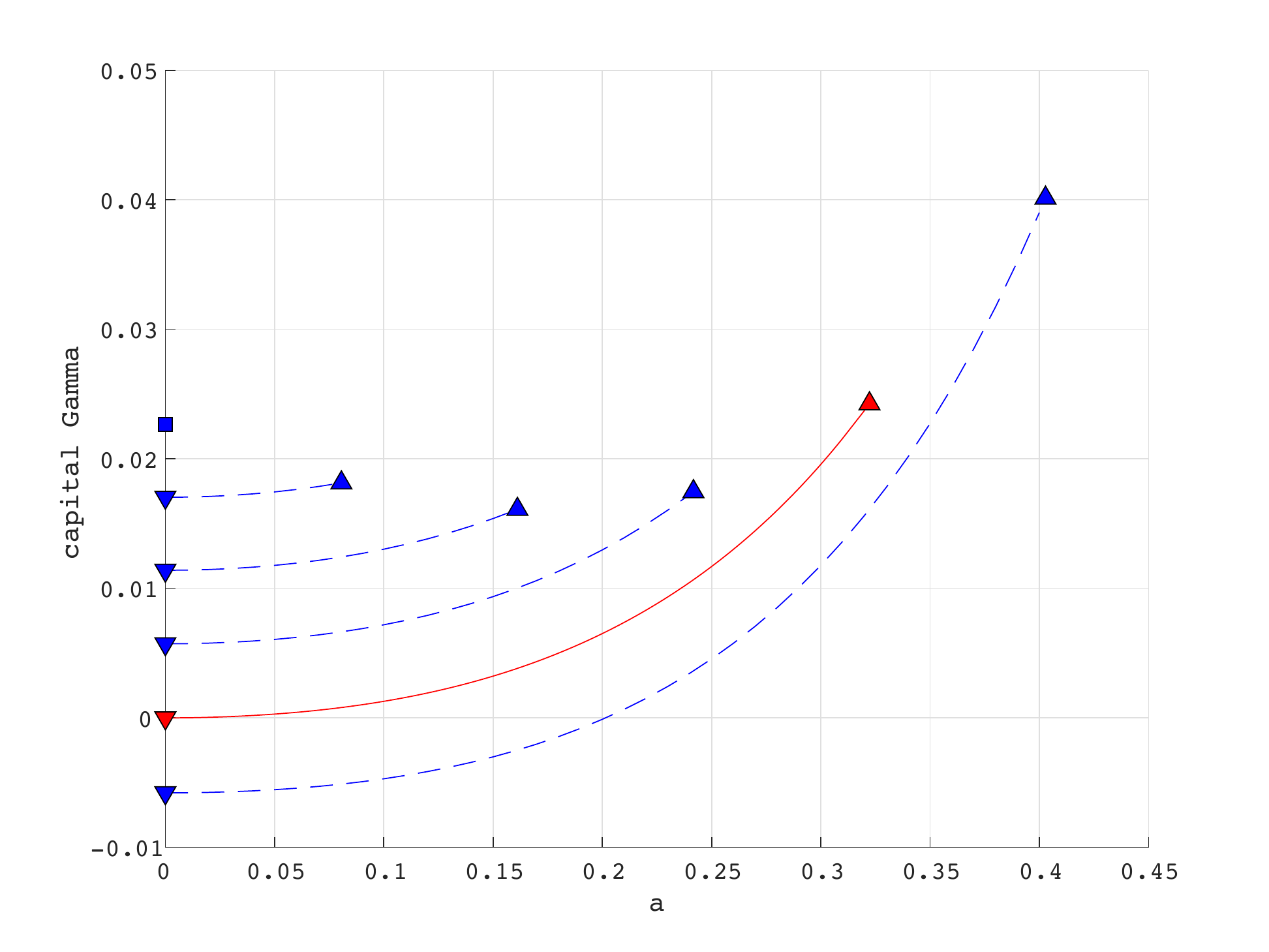}  \\
   \includegraphics[scale=0.33]{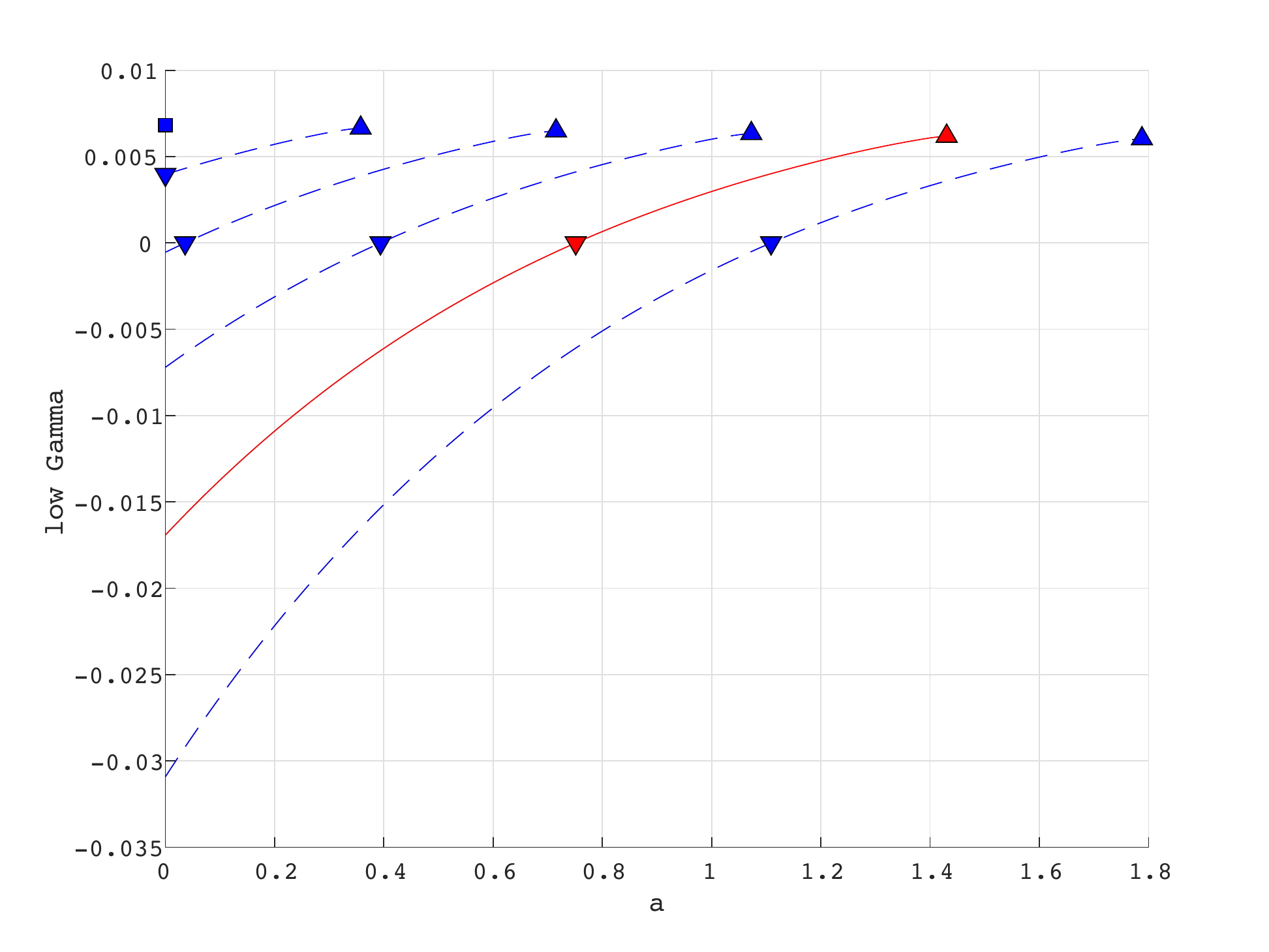} & \includegraphics[scale=0.33]{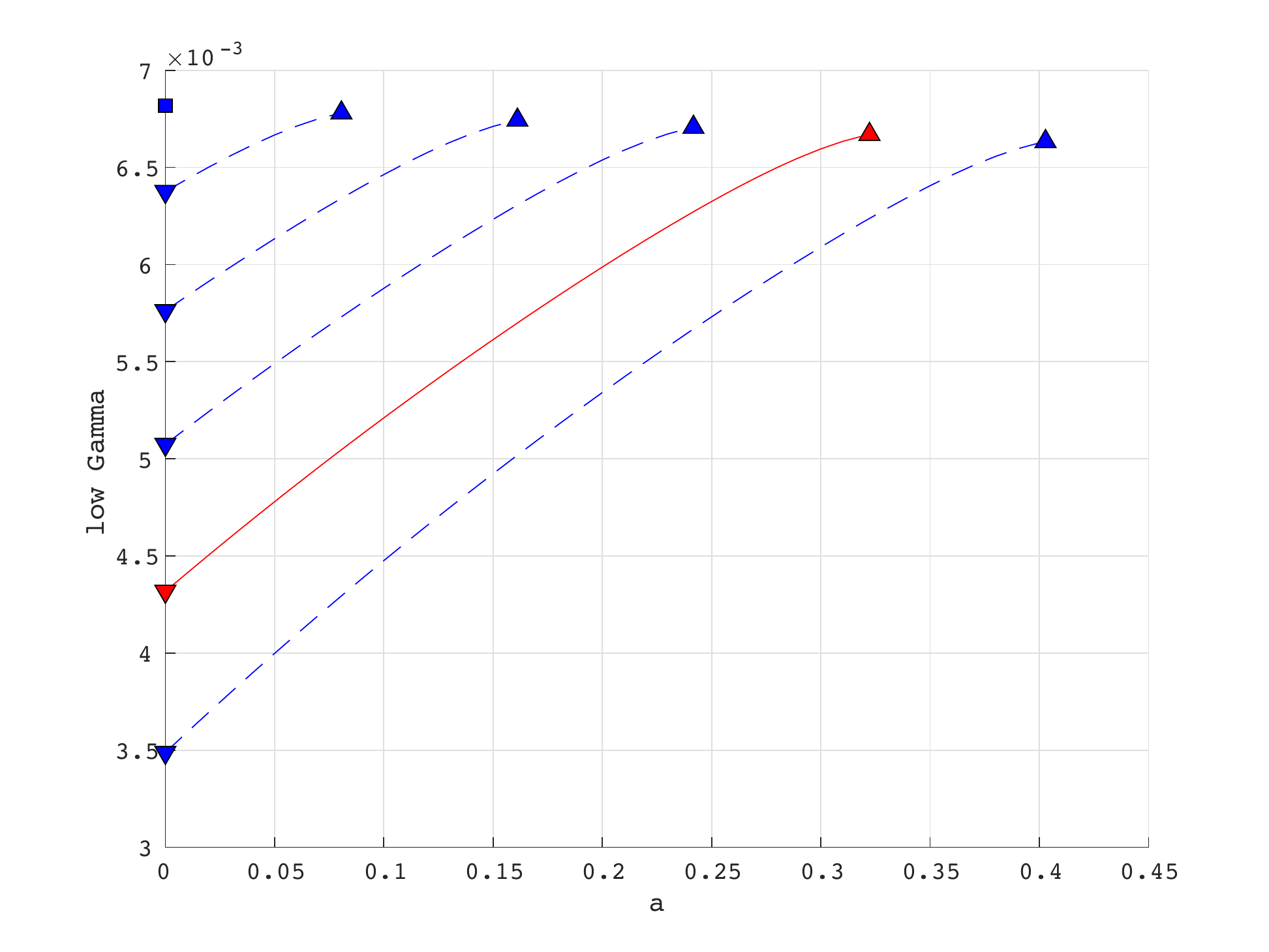}    \\  
\textbf{Case 1 }  ($\rho = 0$) & \textbf{Case 1 }  ($\rho = 5$) \end{tabular}
\end{minipage}
\caption{Plots of $a \mapsto \Gamma(a, b)$ (top) and $a \mapsto \gamma(a, b)$ (bottom)  for \textbf{Case 1} (left) and \textbf{Case 2} (right) for $b = 0, b^*/4, b^*/2, 3b^*/4, b^*, 5b^*/4$. 
The solid lines in red correspond to the cases for $b = b^*$.  For $b > 0$, the down- and up-pointing triangles indicate the points at 
$a = a(b)$ and $a=b$, respectively.  For $b = 0$, the points $(0, \Gamma(0,0))$ and $(0,\gamma(0,0))$ are indicated by squares.} \label{plot_gamma}
\end{center}
\end{figure}
\newpage
In Figure \ref{plot_value_function}, we plot the corresponding value functions $v_{a^*, b^*}$ along with suboptimal NPVs 
$v_{a,b}$  given in Lemma \ref{value_fun_explicit}  with $(a,b) \neq (a^*, b^*)$.  
It can be confirmed in both \textbf{Cases 1}  and \textbf{2} that $v_{a^*, b^*}$ dominates $v_{a,b}$ uniformly in $x$ for $(a,b) \neq (a^*, b^*)$.
  \begin{figure}[h!] 
\begin{center}
\begin{minipage}{1.0\textwidth}
\centering
\begin{tabular}{cc}
    \includegraphics[scale=0.4]{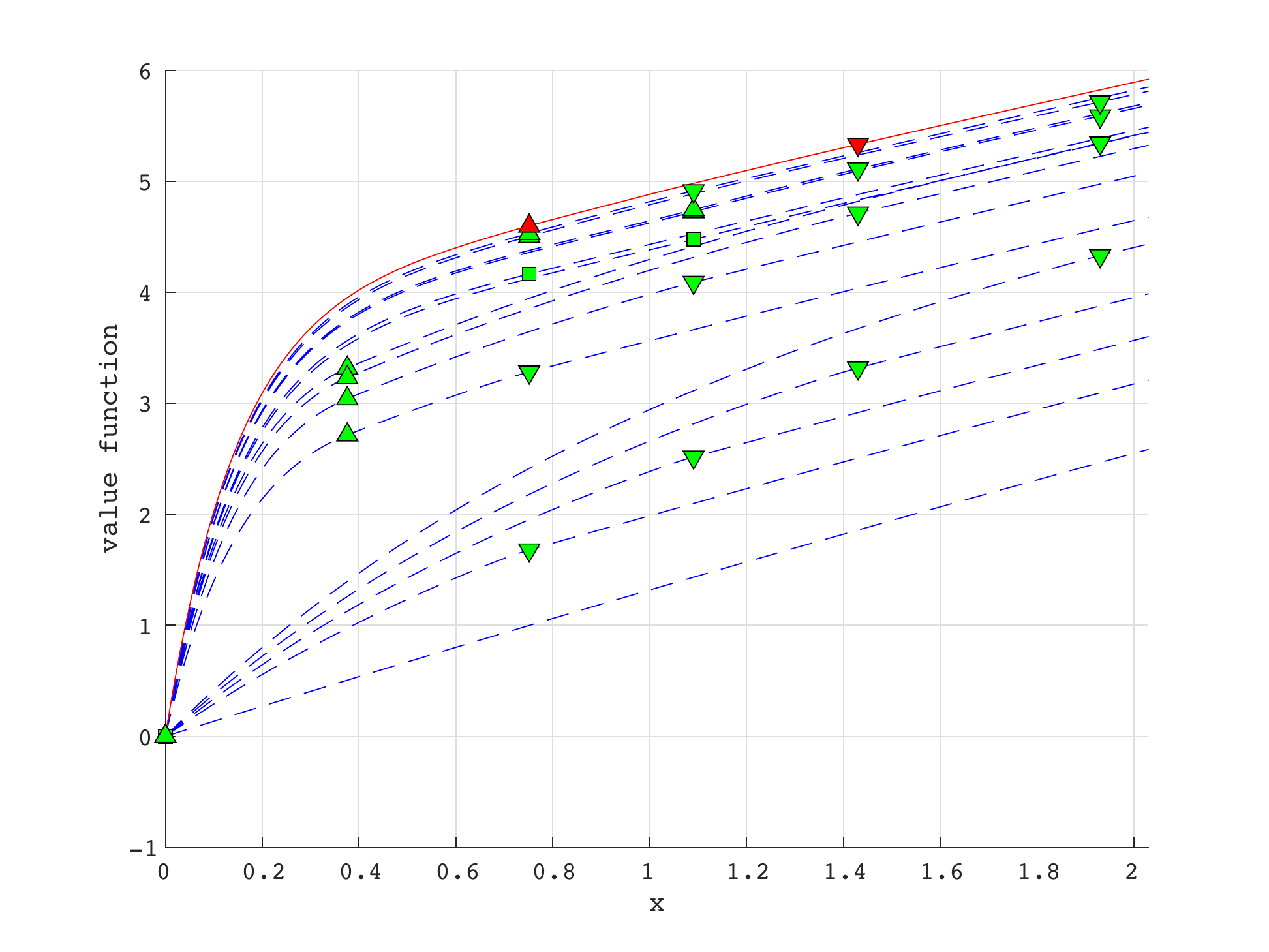} & \includegraphics[scale=0.4]{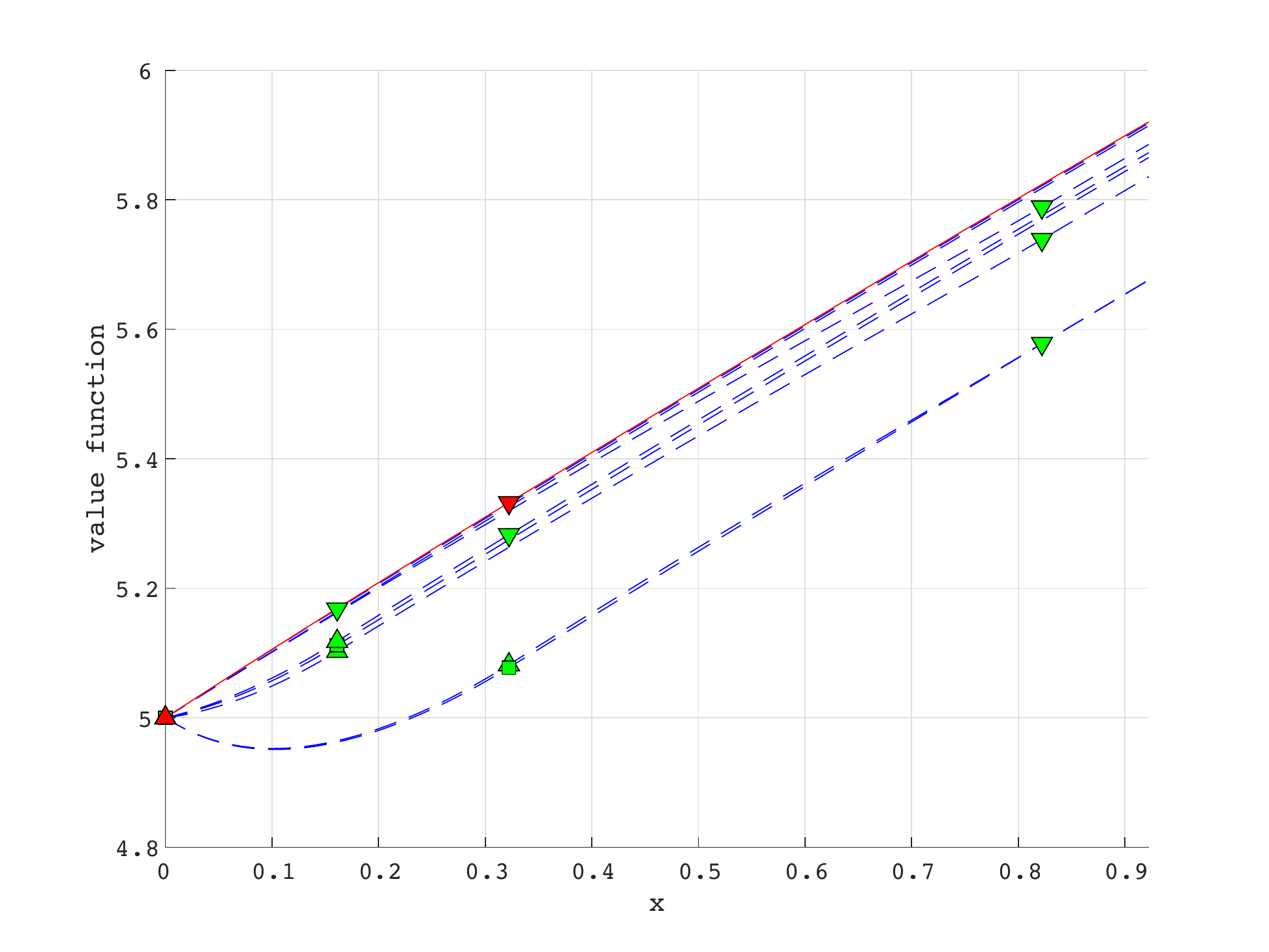} \\
\textbf{Case 1}  ($\rho = 0$) & \textbf{Case 1}  ($\rho = 5$)  \end{tabular}
\end{minipage}
\caption{Plots of the value functions $v_{a^*,b^*}$ (solid) along with the NPVs under suboptimal threshold strategies (dotted) with 
$(a, b) \neq (a^*, b^*)$. Suboptimal NPVs $v_{a,b}$ are plotted for the pairs of $a = 0,a^*/2, a^*, (a^*+b^*)/2$ and $b = 0, a^*, 
(a^*+b^*)/2, b^*, b^*+0.5$ for which $b \geq a$ for \textbf{Case 1}  and $a = 0,b^*/2, b^*$ and $b = 0, b^*/2, b^*,b^*+0.5, b^*, b^*+0.5$ 
for which $b \geq a$ for \textbf{Case 2}. The up- and down-pointing triangles indicate the values at $a^*/a$ and $b^*/b$, respectively.
} \label{plot_value_function}
\end{center}
\end{figure}

\subsection{Sensitivity and convergence}
We next study the behavior of the value function $v_{a^*, b^*}$ with respect to the parameters that describe the problem.  Here we use the same parameters as \textbf{Case 1} unless stated otherwise.

\subsubsection{With respect to $\rho$}

We first study the sensitivity with respect to the terminal payoff $\rho$.   In Figure \ref{plot_rho}, we plot $v_{a^*, b^*}$ along with the optimal thresholds $a^*$ and $b^*$ for $\rho$ ranging from $-6$ to $6$.
It is clear that the value function $v_{a^*, b^*}$ is monotonically increasing in $\rho$.  When $\rho$ is large ($\rho = 6$), we have $a^* = b^* = 0$ meaning that it is optimal to pay dividend as much as possible to enjoy quickly the terminal payoff.  For the case $\rho = 5$, we have $0 = a^* < b^*$ and, for the smaller values of $\rho$, we have $0 < a^* < b^*$.  We see that both $a^*$ and $b^*$ increase as $\rho$ decreases.

  \begin{figure}[h!]
\begin{center}
\begin{minipage}{1.0\textwidth}
\centering
\begin{tabular}{c}
 \includegraphics[scale=0.4]{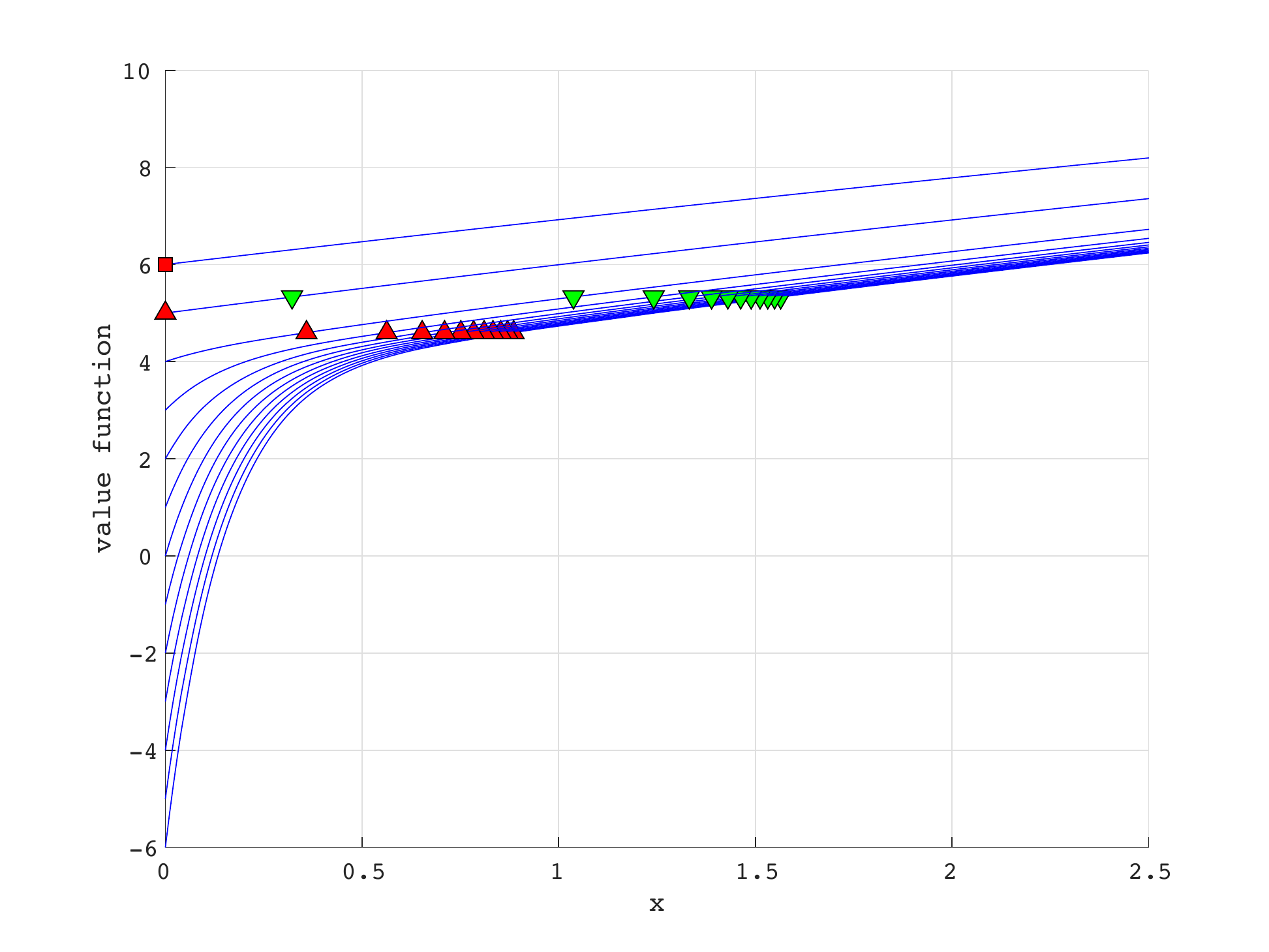}  \end{tabular}
\end{minipage}
\caption{Plots of the value function $v_{a^*,b^*}$ for $\rho = -6, -5, \ldots, 5, 6$. For the case $\rho = 6$, we have $a^*=b^* = 0$ and its value there is indicated by the square.  For other values of $\rho$, we have $0 \leq a^* < b^*$ and the points at $a^*$ and $b^*$ are indicated by the up- and down-pointing triangles, respectively.} \label{plot_rho}
\end{center}
\end{figure}
\newpage
\subsubsection{With respect to $\beta$}
We next analyze the behaviors with respect to the unit cost of capital injection $\beta$. 
As $\beta \uparrow \infty$, it becomes costly to inject capital and hence we expect that $a^* \downarrow 0$. Consequently, 
the value function approaches (when $\rho = 0$) to that of Yin et al.\ \cite{YSP}  with the absolutely continuous assumption without capital injections: 
its optimal strategy is of threshold-type with the threshold $b^*_{\infty, \delta_1, \delta_2,0}$ such that 
\begin{align*}
\delta_1 Z_1^{(q)}(b^*_{\infty, \delta_1, \delta_2,0})- \frac q {\Phi_0(q)} e^{\Phi_0(q) a}\Big( 1  +\delta_1 \Phi_0(q) \int_{0}^{b^*_{\infty, \delta_1, \delta_2,0}}Z_1^{(q)}(y) e^{-\Phi_0(q) y} \ud y \Big) = 0,
\end{align*}
and the value function is given by 
\begin{align*} 
v(x; \infty, \delta_1, \delta_2,\rho=0) &= - e^{\Phi_0(q) (b^*_{\infty, \delta_1, \delta_2,0}-x)} 
\delta_1 \int_0^{b^*_{\infty, \delta_1, \delta_2,0}-x} W_1^{(q)} (z) e^{-\Phi_0(q) z} \diff z \\ 
&+ \frac {\delta_1} q Z_1^{(q)}(
b^*_{\infty, \delta_1, \delta_2,0}-x) - \frac {e^{\Phi_0(q)(b^*_{\infty, \delta_1, \delta_2,0}-x)}} {\Phi_0(q)}, \quad x \geq 0. 
\end{align*}


On the other hand, as $\beta \downarrow 1$  and $\rho=0$, the expected NPV of payoffs \eqref{v_pi} converges to 
(with $\tilde{r}^\pi := \delta_2 - r^\pi$) 

\begin{align*}
	\mathbb{E}_x \left( \int_0^{\kappa_0^{\pi}} e^{-q t} (\ell^{\pi}(t)- r^{\pi}(t))   \diff t \right)
	&= \mathbb{E}_x \left( \int_0^{\kappa_0^{\pi}} e^{-q t} (\ell^{\pi}(t) + \tilde{r}^{\pi}(t) - \delta_2)  \diff t   \right) \\
		&= \mathbb{E}_x \left( \int_0^{\kappa_0^{\pi}} e^{-q t} (\ell^{\pi}(t) + \tilde{r}^{\pi}(t))  \diff t  + \frac{\delta_2}{q}e^{-q \kappa_0^\pi}\right) - \frac{\delta_2}{q},
\end{align*}
with $0 \leq \ell^{\pi} + \tilde{r}^{\pi} \leq \delta_1 + \delta_2$.  
Its maximization corresponds to Yin et al.\ \cite{YSP} with an additional terminal payoff $\delta_2/q$. 
 It is conjectured that the optimal strategy is of threshold-type. Consequently, we expect  that $(b^* - a^*) \rightarrow 0$ as $\beta \downarrow 1$.


In Figure \ref{plot_beta}, we plot $v_{a^*, b^*}$ for various values of $\beta$ along with the limiting cases $v(\cdot; \infty, \delta_1, \delta_2,\rho=0)$.  It is confirmed that the value function $v_{a^*, b^*}$ is 
monotonically decreasing in $\beta$.  
In addition, we see that, as $\beta \uparrow \infty$, $v_{a^*, b^*}$ converges decreasingly to $v(\cdot; \infty, \delta_1, \delta_2,\rho=0)$ with $a^* \downarrow 0$ and 
$b^* \uparrow b_{\infty, \delta_1, \delta_2,0}^*$. On the other hand, as $\beta \downarrow 1$, we see that both $a^*$ and $b^*$ converge to the same value.

  \begin{figure}[h!]
\begin{center}
\begin{minipage}{1.0\textwidth}
\centering
\begin{tabular}{c}
 \includegraphics[scale=0.4]{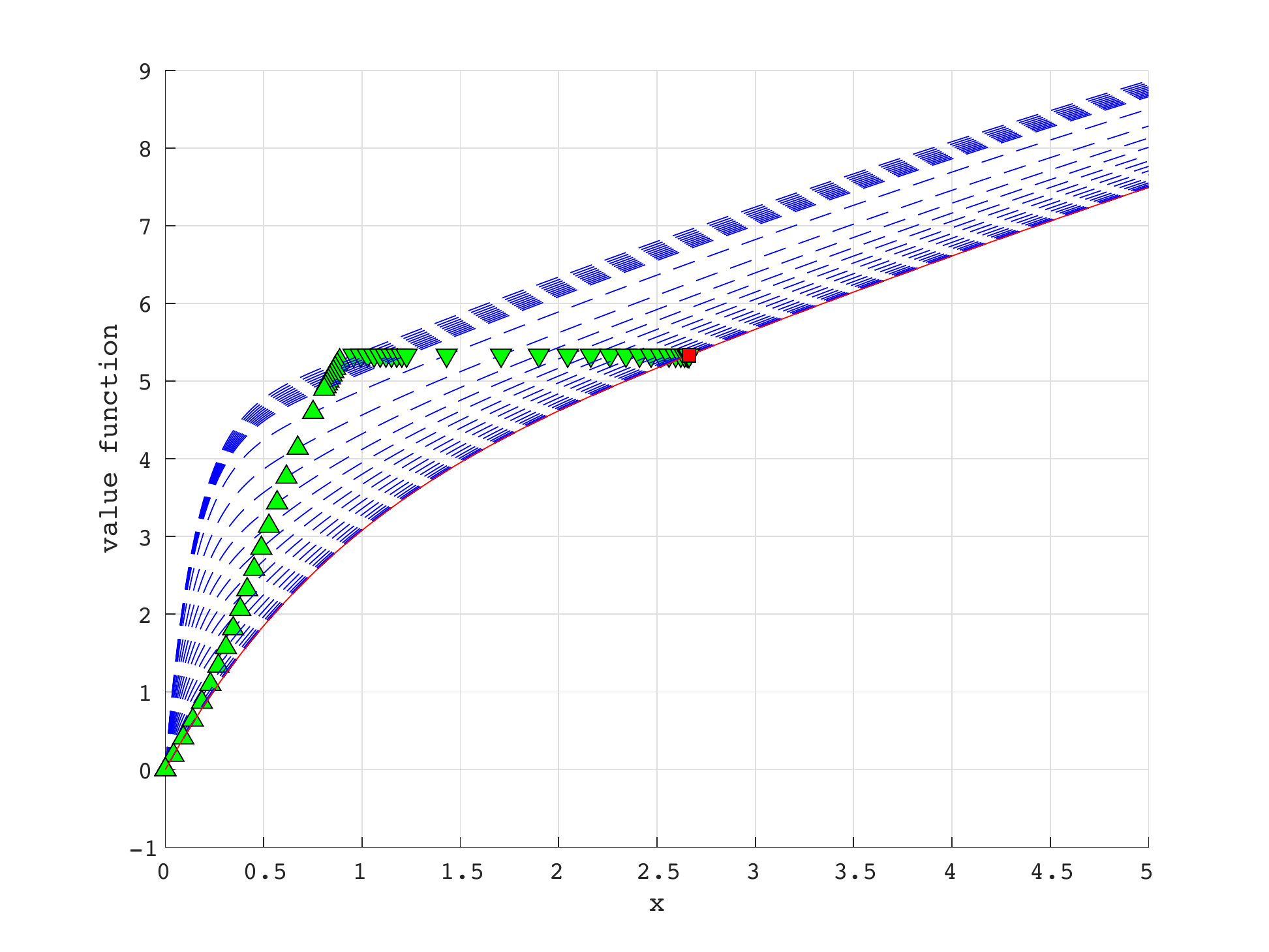}  \end{tabular}
\end{minipage}
\caption{Plots of the value function $v_{a^*,b^*}$ for $\beta = 1.01$, $1.02$, $\ldots$, $1.09$, $1.1$, $1.2$, $1.4$, $\ldots$, $4.8$, $5$ (dotted) along with $v(\cdot; \infty, \delta_1, \delta_2,\rho=0)$ 
(solid). The points at $a^*$ and $b^*$ are indicated by the up- and down-pointing triangles, respectively. The point at $b_{\infty, \delta_1, \delta_2,0}^*$ is indicated by a square.}  \label{plot_beta}
\end{center}
\end{figure}

\subsubsection{With respect to $\delta_2$}

We now analyze the behaviors with respect to $\delta_2$. 
The case $\delta_2 = 0$ corresponds to Yin et al.\ \cite{YSP} where, as reviewed above, its value function is 
$v(\cdot; \beta, \delta_1, 0,\rho=0) = v(\cdot; \infty, \delta_1, \delta_2,\rho=0)$ with the optimal barrier $b^*_{\beta, \delta_1, 0,0} = b^*_{\infty, \delta_1, \delta_2,0}$.  
On the other hand, as $\delta_2 \uparrow \infty$, there are two possible scenarios either to let it liquidate (\textbf{Case L}) 
or to reflect the surplus to avoid ruin (\textbf{Case R}). 
It is easy to see that \textbf{Case L} is equivalent to the case $\delta_2 = 0$; 
in this case, because $v_{a^*,b^*}$ is monotonically increasing 
in $\delta_2$, we have $v_{a^*, b^*} =  v(\cdot; \beta, \delta_1, 0,\rho=0)$ for any choice of $\delta_2$.
\textbf{Case R} corresponds to P\'erez and Yamazaki \cite{YP_RR_dual} with the absolutely continuous assumption on the dividend and classical capital injection: its value function is given by 
\begin{align*} 
\begin{split}
v_R(x; \beta, \delta_1, \infty,\rho=0) &=  \delta_1  \frac {Z_1^{(q)}(b^{*,R}_{\beta, \delta_1, \infty,0}-x)} q -  e^{\Phi_0(q) b^{*,R}_{\beta, \delta_1, \infty,0}}  
\Big( \frac {e^{-\Phi_0(q)x}} {\Phi_0(q)}+\delta_1 l^{(q)}(x; b^{*,R}_{\beta, \delta_1, \infty,0}) 
\Big),
\end{split}
\end{align*}
 with the optimal barrier $b^{*,R}_{\beta, \delta_1, \infty,0}$ given as the root of $e^{-\Phi_0(q)b}  \beta =1  +\delta \Phi_0(q) l^{(q)}(0; b)
 $. 

The left plot of Figure \ref{plot_delta2} shows the convergence of $v_{a^*, b^*}$ for \textbf{Case R} 
(the used parameters are the same as \textbf{Case 1} ). Here, we plot $v_{a^*, b^*}$ for various values of $\delta_2$ along with the limiting case 
$v(x; \beta, \delta_1, 0,\rho=0)$ and $v_R(x; \beta, \delta_1, \infty,\rho=0)$. We see that, as $\delta_2 \downarrow 0$, $v_{a^*,b^*}$ converges decreasingly to $v(x; \beta, \delta_1, 0,\rho=0)$ with  
$b^* \uparrow b^*_{\beta, \delta_1, 0,0}$.  On the other hand, as $\delta_2 \uparrow \infty$, it converges increasingly to $v_R(x; \beta, \delta_1, \infty,\rho=0)$ with $a^* \downarrow 0$ and 
$b^* \downarrow b^{*,R}_{\beta, \delta_1, \infty,0}$.
The right plot of Figure \ref{plot_delta2} illustrates \textbf{Case L} (here we use the same parameters as \textbf{Case 1} except that $\beta = 6$) 
and shows $v_{a^*, b^*}(x) = v(x; \beta, \delta_1, 0,\rho=0)$ in comparison to  $v_R(x; \beta, \delta_1, \infty,\rho=0)$.
  \begin{figure}[h!]
\begin{center}
\begin{minipage}{1.0\textwidth}
\centering
\begin{tabular}{cc}
 \includegraphics[scale=0.4]{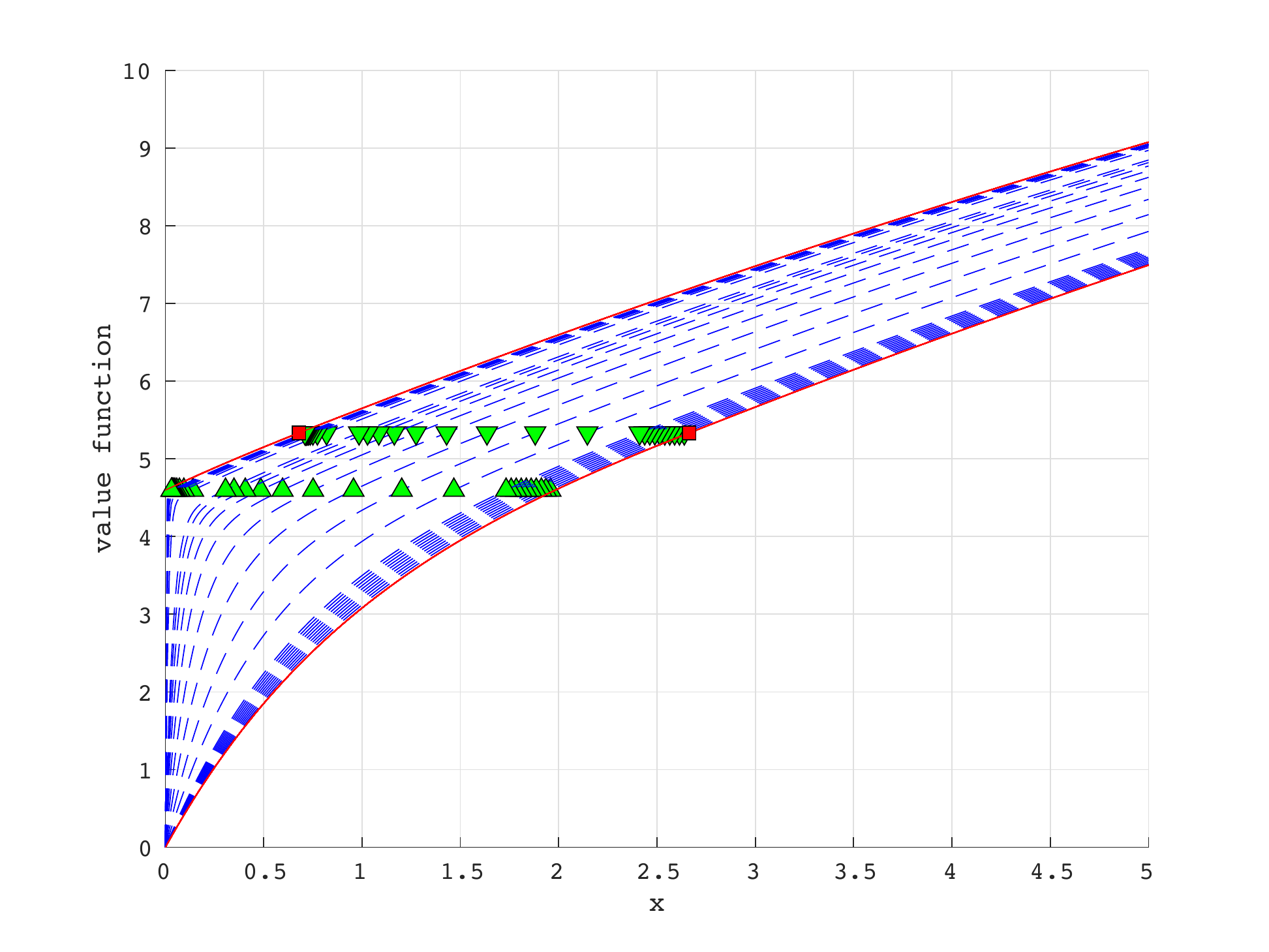} &  \includegraphics[scale=0.4]{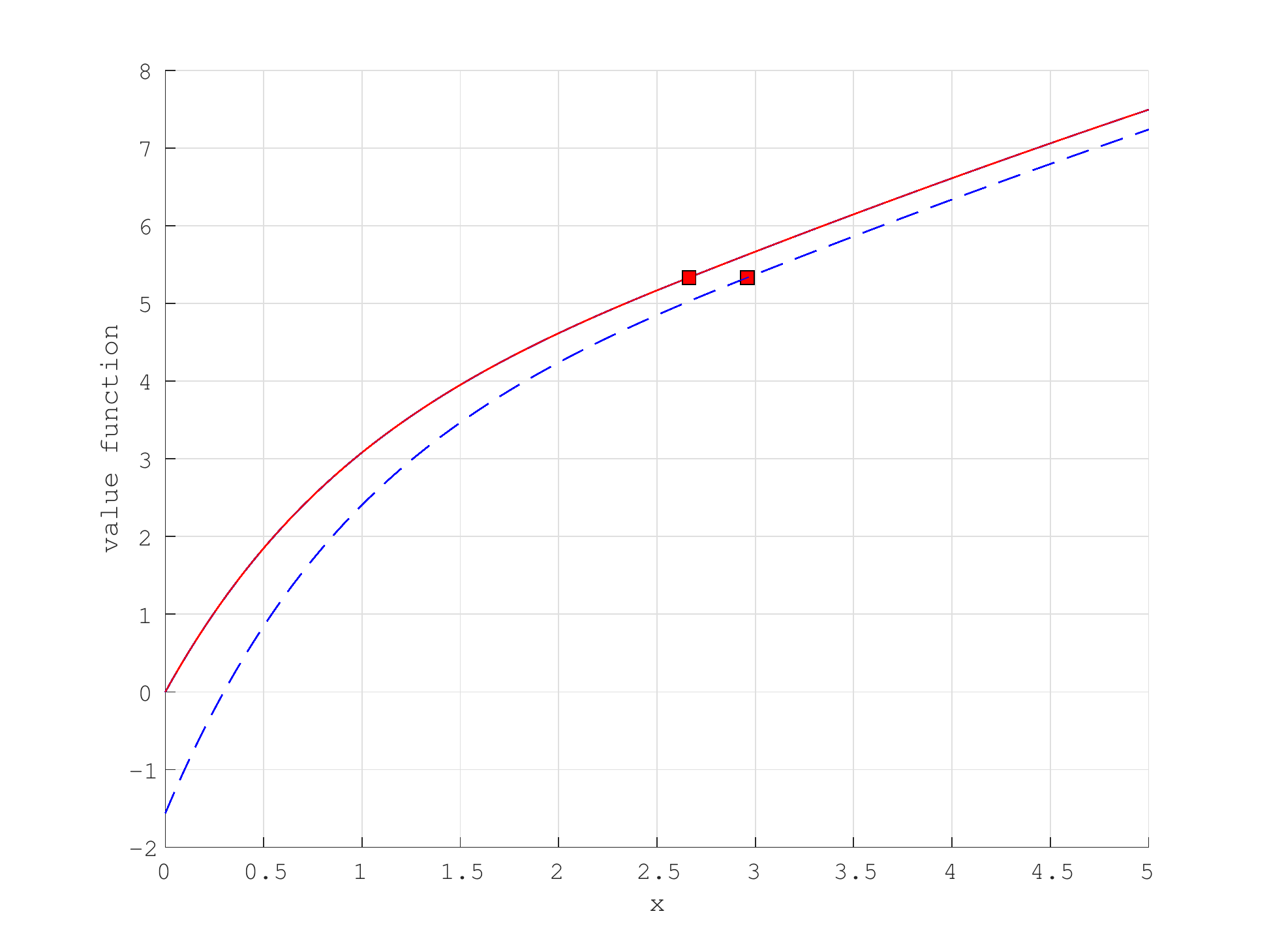}  \\
 Case R & Case L \end{tabular}
\end{minipage}
\caption{(Left) Plots of the value function $v_{a^*,b^*}$ for $\delta_2 = 0.01$, $0.02$, $\ldots$, $0.09$, $0.1$, $0.2$, $\ldots$, $0.9$, $1$, $2$, $\ldots$, $9$, $10$ along with the limiting cases 
$v(\cdot; \beta, \delta_1, 0,\rho=0)$ and $v_R(\cdot; \beta, \delta_1, \infty,\rho=0)$. 
The points at $a^*$ and $b^*$ are indicated by the up- and down-pointing triangles, respectively, and $b^*_{\beta, \delta_1, 0,0}$
 and $b^{*,R}_{\beta, \delta_1, \infty,0}$  are indicated by squares. (Right) Plots of $v_{a^*,b^*} = v(\cdot; \beta, \delta_1, 0,\rho=0)$ (solid)
 and $v_R(\cdot; \beta, \delta_1, \infty,\rho=0)$ (dotted). 
The points at $b^{*}_{\beta, \delta_1, 0,0}$ and $b^{*,R}_{\beta, \delta_1, \infty,0}$  are indicated by squares.}  \label{plot_delta2}
\end{center}
\end{figure}

\subsubsection{With respect to $\delta_1$}

For the analysis on $\delta_1$, we consider Problem 2 that we described in Section \ref{subsection_problem2} so as to analyze the convergence to the hybrid refracted-reflected case in Avanzi et al.\ \cite{APWY}.

Let $\tilde{v}_{\tilde{a}^*, \tilde{b}^*}$ be the value function for Problem 2 (which is equivalent to the maximal value of \eqref{problem_2_equivalent} with the optimal thresholds 
$(\tilde{a}^*, \tilde{b}^*)$).  For the case $\delta_1 = 0$, it reduces again to Yin et al.\ \cite{YSP}. 
The value function becomes $\beta$ times the one reviewed above with $\delta_1$ replaced with $\delta_2$: $\tilde{v}(\cdot; \beta, 0, \delta_2,\rho=0) = \beta v(\cdot; \beta, \delta_2, 0,\rho=0)$ 
with the optimal threshold $\tilde{b}^*_{\beta, 0, \delta_2,0} = b^*_{\beta, \delta_2, 0,0}$.
%
On the other hand, as $\delta_1 \uparrow \infty$, the problem gets closer to the problem considered in Avanzi et al.\ \cite{APWY} 
where a strategy is a combination of the absolutely continuous and singular control with transaction costs $\beta$ and $1$, respectively.  
We refer the reader to Avanzi et al.\ \cite{APWY} for the form of the value function $\tilde{v}(\cdot; \beta, \infty, \delta_2,\rho=0)$ 
and the characterization of the optimal barriers $(\tilde{a}^*_{\beta, \infty, \delta_2,0},\tilde{b}^*_{\beta, \infty, \delta_2,0})$.

In Figure \ref{plot_delta1}, we plot $\tilde{v}_{\tilde{a}^*, \tilde{b}^*}$ for various values of $\delta_1$ along with $\tilde{v}(x; \beta, 0, \delta_2,\rho=0)$ and $\tilde{v}(x; \beta, \infty, \delta_2,\rho=0)$. It is confirmed that $\tilde{v}_{\tilde{a}^*,\tilde{b}^*}$ is increasing in $\delta_1$.
We see that, as $\delta_1 \downarrow 0$, the value function converges decreasingly to the limiting case $\tilde{v}(x; \beta, 0, \delta_2,\rho=0)$ with  $\tilde{a}^* \uparrow \tilde{b}^*_{\beta, 0, \delta_2,0}$. 
 On the other hand, as $\delta_1 \uparrow \infty$, the value function converges increasingly to $\tilde{v}(x; \beta, \infty, \delta_2,\rho=0)$ with 
$\tilde{a}^* \rightarrow \tilde{a}^*_{\beta, \infty, \delta_2,0}$ and $\tilde{b}^* \rightarrow \tilde{b}^*_{\beta, \infty, \delta_2,0}$.

  \begin{figure}[h!]
\begin{center}
\begin{minipage}{1.0\textwidth}
\centering
\begin{tabular}{c}
 \includegraphics[scale=0.4]{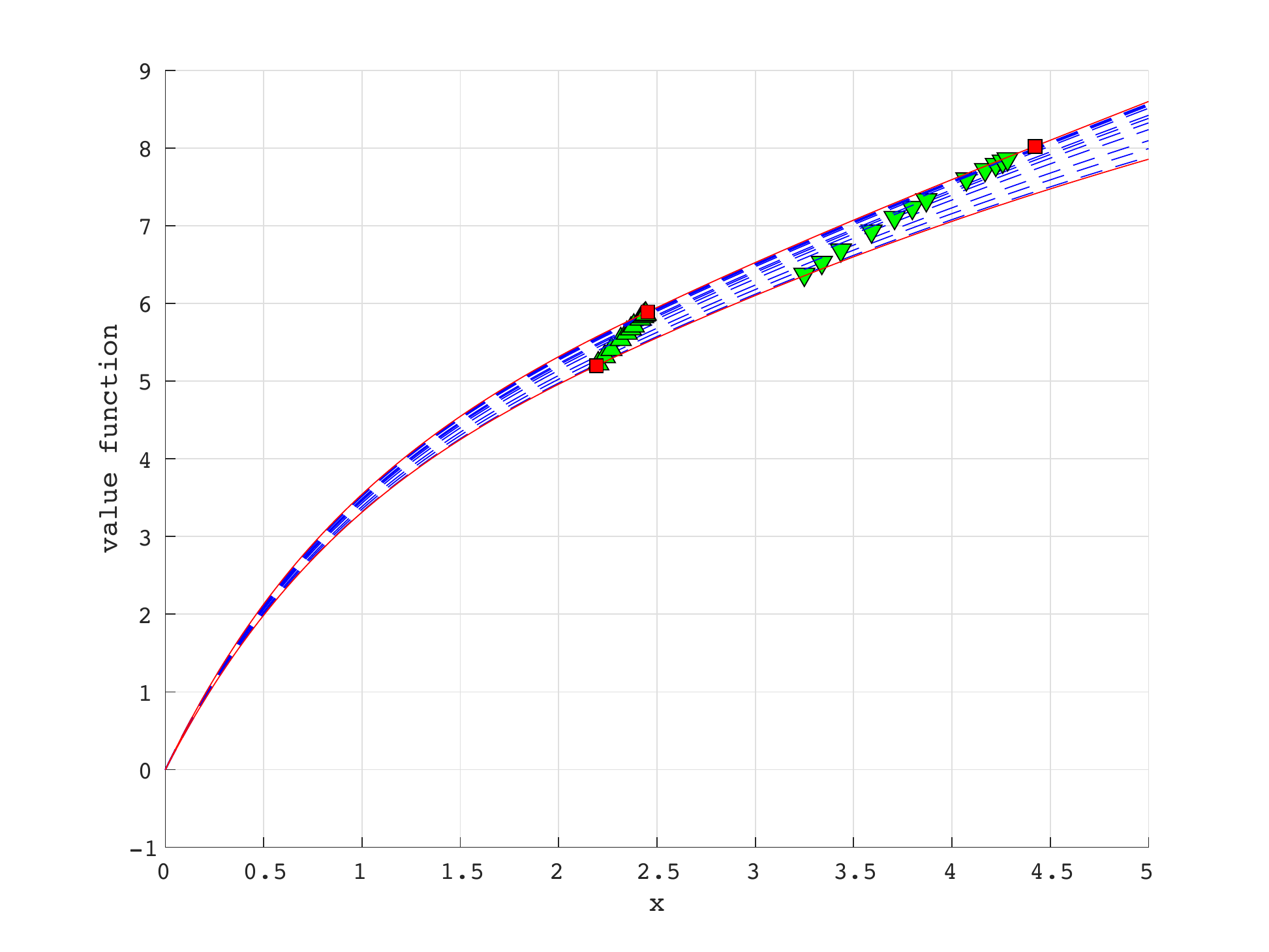}  \end{tabular}
\end{minipage}
\caption{Plots of the value function $\tilde{v}_{a^*,b^*}$ for $\delta_1 = 0.01$, $0.05$, $0.1$, $0.2$, $0.3$, $0.4$, $0.5, 1.0, \ldots, 2.5,3$ (dotted) along with 
$\tilde{v}(x; \beta, 0, \delta_2,\rho=0)$ and $\tilde{v}(x; \beta, \infty, \delta_2,\rho=0)$ (solid). 
The points at $\tilde{a}^*$ and $\tilde{b}^*$ are indicated by the up- and down-pointing triangles, respectively. The points at $\tilde{a}_{\beta, 0, \delta_2,0}^*$, 
$\tilde{a}_{\beta, \infty, \delta_2,0}^*$, and $\tilde{b}_{\beta, \infty, \delta_2,0}^*$ are indicated by squares.} \label{plot_delta1}
\end{center}
\end{figure}

%
%
%
\begin{appendix}

\section{Proof of Lemma \ref{value_fun_explicit}} \label{appendix_proof_value_fun_explicit}
Throughout this Appendix, we fix $0 \leq a < b$.
Let us define, for $y \in \R$, 
	\begin{align} \label{def_u}
	\begin{split}
	u_1^{(q)}(y) &:= e^{\Phi_0(q) y} + \delta_1 \Phi_0(q) l^{(q)}(-y;b), \\
	u_2^{(q)}(y) &:= u_{1}^{(q)}(y)+ \delta_2 \int_{-a}^y W_2^{(q)}(y-z) u^{(q)\prime}_{1} (z) \diff z,
	\end{split}
	\end{align}
	and, for $d \geq 0$,
	 \begin{align}\label{def_w}
	 	\begin{split}
	 w_1^{(q)}(y; d) &:= W_0^{(q)} (y-d) + \delta_1 \int_{-b}^{y} W_1^{(q)} (y-u) W^{(q)\prime}_0 (u-d) \diff u, \\
		w_2^{(q)}(y;d)&:=w_1^{(q)}(y;d)+\delta_2\int_{-a}^{y}W_2^{(q)}(y-u)w_1^{(q)\prime}(u;d) \diff u.
	\end{split}
	 \end{align}

Let $\tilde{V}^{-b,-a} := - V^{a,b}$ as in \eqref{V_a_b_SDE} and $\tilde{\kappa}_0^+ := \inf \{ t > 0: \tilde{V}^{-b,-a}(t) > 0 \}$. 

First, by Theorem 3 in \cite{CPRY}, we have 
\begin{align} \label{upcrossing_laplace}
\E_x\left(e^{-q\kappa_0^-}\right) = \E_{-x}\left(e^{-q \tilde{\kappa}_0^+} \right) 
=\frac{u_2^{(q)}(-x)}{u_2^{(q)}(0)}, \quad x \geq 0.
\end{align}

For the expected NPV of dividends,  Theorem 4 in \cite{CPRY} and \eqref{upcrossing_laplace} give
\begin{align*}
	\mathbb{E}_x &\left( \int_0^{\kappa_0^-} e^{-q t} \diff L_t^{a,b}\right)=
	\mathbb{E}_{x} \left( \int_0^{\kappa_0^-} e^{-q t} \delta_11_{\{V^{a,b}_t \geq b\}}\diff t\right)=
	\mathbb{E}_{-x} \left( \int_0^{\tilde{\kappa}_0^+} e^{-q t} \delta_11_{\{\tilde{V}^{-b,-a}_t \leq -b\}} \diff t\right)\\
	&=\frac{\delta_1}{q}\left(1-\E_{-x}\left(e^{-q \tilde{\kappa}_0^+};\tilde{\kappa}_0^+<\infty\right) \right)-
	\delta_1\mathbb{E}_{-x} \left( \int_0^{\tilde{\kappa}_0^+} e^{-q t}1_{\{\tilde{V}^{-b,-a}_t>-b\}} \diff t\right) \\
	&=\frac{\delta_1}{q}\left(1-\frac{u_2^{(q)}(-x)}{u_2^{(q)}(0)}\right)-\frac{\delta_1}{1-\delta_1 W^{(q)}_0(0)}\int_{-b}^{-a}\left\{\frac{u_2^{(q)}(-x)}{u_2^{(q)}(0)} w_2^{(q)}(0;y)- w_2^{(q)}(-x;y)\right\} \diff y\\
	&-\frac{\delta_1}{{\prod_{j=1}^2\left(1-\delta_j W_{j-1}^{(q)}(0)\right)}}\int_{-a}^{0}\left\{\frac{u_2^{(q)}(-x)}{u_2^{(q)}(0)}w_2^{(q)}(0;y)-w_2^{(q)}(-x;y)\right\} \diff y.
\end{align*}
\par On the other hand, for the expected NPV of capital injections, by Theorem 4 in \cite{CPRY}, 
\begin{align*}
	\mathbb{E}_x \left( \int_0^{\kappa_0^-} e^{-q t} \diff R_t^{a,b}\right)&=\mathbb{E}_x \left( \int_0^{\kappa_0^-} e^{-q t} 
	\delta_21_{\{V^{a,b}_t < a\}}\diff t\right)
	=\mathbb{E}_{-x} \left( \int_0^{\tilde{\kappa}_0^+} e^{-q t} \delta_21_{\{\tilde{V}^{-b,-a}_t> -a \}} \diff t \right) \\
	&=\frac{\delta_2}{{\prod_{j=1}^2\left(1-\delta_j W_{j-1}^{(q)}(0)\right)}}\int_{-a}^{0} \left\{\frac{u_2^{(q)}(-x)}{u_2^{(q)}(0)}w_2^{(q)}(0;y)-w_2^{(q)}(-x;y)\right\}\diff y.
\end{align*}

Combining these, the expected NPV \eqref{v_pi}, for $x \geq 0$, becomes
\begin{align}\label{vf_2}
\begin{split}
	v_{a,b}(x)&=\frac{\delta_1}{q}-\frac{\delta_1}{1-\delta_1 W^{(q)}_0(0)}\int_{-b}^{-a}\left\{\frac{u_2^{(q)}(-x)}{u_2^{(q)}(0)}w_2^{(q)}(0;y)-w_2^{(q)}(-x;y) \right\} 
	\diff y \\&-\frac{(\delta_1+\beta\delta_2)}{{\prod_{j=1}^2\left(1-\delta_j W_{j-1}^{(q)}(0)\right)}}\int_{-a}^{0} \left\{\frac{u_2^{(q)}(-x)}{u_2^{(q)}(0)}w_2^{(q)}(0;y)-w_2^{(q)}(-x;y)\right\}\diff y
	-\left(\frac{\delta_1}{q}-\rho\right)\frac{u_2^{(q)}(-x)}{u_2^{(q)}(0)} \\
	&=- \frac {u_2^{(q)}(-x)}{u_2^{(q)}(0)} (\tilde{f}^{(q)}_{a,b}(0)-\rho) +  \tilde{f}^{(q)}_{a,b}(-x)	\end{split}
	\end{align}
	where
	\begin{align} \label{def_f_tilde}
	\tilde{f}^{(q)}_{a,b}(z) :=  \frac{\delta_1}{q}+ \frac{\delta_1 \int_{-b}^{-a} w_2^{(q)}(z;y)  \diff y }{1-\delta_1 W_0^{(q)}(0)} +\frac{(\delta_1+\beta\delta_2) \int_{-a}^{0} w_2^{(q)}(z;y) \diff y}
	{{\prod_{j=1}^2\left(1-\delta_j W_{j-1}^{(q)}(0)\right)}}, \quad z \leq 0.
	\end{align}

In view of \eqref{vf_2},  the proof of  Lemma \ref{value_fun_explicit} is complete  once we show Lemmas \ref{lemma_u} and \ref{lemma_f} below.

\begin{lemma} \label{lemma_u} For $x \geq 0$, we have
	$u_2^{(q)}(-x) = g_{a,b}^{(q)}(x)$.
\end{lemma}
 \begin{proof} 

Applying $u_1^{(q)\prime}(y) = \Phi_0(q) e^{\Phi_0(q) y} - \delta_1 \Phi_0(q) l^{(q)\prime}(-y;b)$ in \eqref{def_u} and changing variables, we obtain the result.
\end{proof}

\begin{lemma} \label{lemma_f} 
	For $x \geq 0$, we have $\tilde{f}^{(q)}_{a,b}(-x) = f_{a,b}^{(q)}(x)$.	
\end{lemma}
\begin{proof}
By Lemma 14 in \cite{CPRY},
\begin{align}\label{iden_1}
	&\frac{\int_{-a}^{0}w_2^{(q)}(-x;y) \diff y}{{\prod_{j=1}^2\left(1-\delta_j W_{j-1}^{(q)}(0)\right)}}=\int_{-a}^{0}W_2^{(q)}(-x-y)\diff y = \overline{W}_2^{(q)}(a-x), \\
\label{iden_2}
	&\frac{w_1^{(q)}(-x; y)}{1-\delta_1 W^{(q)}_0(0)}=W_1^{(q)}(-x-y), \quad -b < y \leq -a.
\end{align}

From \eqref{def_w}  and \eqref{iden_2}, together with Fubini's theorem and \eqref{RLqp}, we obtain
	\begin{align} \label{w_2_integral_a_b}
	\begin{split}
		&\frac{\int_{-b}^{-a}w_2^{(q)}(-x;y)\diff y}{1-\delta_1 W^{(q)}_0(0)} =\int_{-b}^{-a}\left(W_1^{(q)}(-x-y)+\delta_2\int_{-a}^{-x}W_2^{(q)}(-x-z)W_1^{(q)\prime}(z-y)\diff z \right)\diff y \\
		&=\overline{W}_1^{(q)}(b-x)-\overline{W}_1^{(q)}(a-x)+\delta_2\int_{-a}^{-x}W_2^{(q)}(-x-z)\left(W_1^{(q)}(b+z)-W_1^{(q)}(a+z)\right)\diff z \\
		&=\overline{W}_1^{(q)}(b-x)-\overline{W}_1^{(q)}(a-x)+\delta_2\int_{-a}^{-x}W_2^{(q)}(-x-z)W_1^{(q)}(b+z)\diff z \\&-\left(\overline{W}_2^{(q)}(a-x)-\overline{W}_1^{(q)}(a-x)\right)\\
		&=\overline{W}_1^{(q)}(b-x)-\overline{W}_2^{(q)}(a-x)+\delta_2\int_{-a}^{-x}W_2^{(q)}(-x-z)W_1^{(q)}(b+z)\diff z.
		\end{split}
	\end{align}

								By substituting \eqref{iden_1} and   \eqref{w_2_integral_a_b} in \eqref{def_f_tilde} and applying change of variables, we have the result.
\end{proof}
\bigskip
\section{Proof of Lemma \ref{verificationlemma}} \label{proof_verificationlemma}
	By the definition of $v$ as a supremum, it follows that $v_{\hat{\pi}}(x)\leq v(x)$ for all $x\in\R$. We write $w:=v_{\hat{\pi}}$ and show that $w(x)\geq v_\pi(x)$ for all $\pi\in\mathcal{A}$ for all $x \geq 0$. 
	
	Fix $\pi\in \mathcal{A}$, $x \geq 0$ and
	let $(T_n)_{n\in\mathbb{N}}$ be the sequence of stopping times defined by $T_n :=\inf\{t>0:V^\pi(t)>n \textrm{ or } {V}^\pi(t)< 1/n \}$. 
	Since ${V}^\pi$ is a semi-martingale and $w$ is sufficiently smooth on $(0, \infty)$ by assumption, we can use the change of variables/Meyer-It\^o's formula 
	(cf.\ Theorems II.31 and II.32 of \cite{protter})  to the stopped process $(e^{-q(t\wedge T_n)}w({V}^\pi(t\wedge T_n)); t \geq 0)$ to deduce under $\mathbb{P}_x$ that 
	\begin{align*}
		e^{-q(t\wedge T_n)}w({V}^\pi(t\wedge T_n))  -w(x)
		&=   \int_{0}^{t\wedge T_n}e^{-qs}   (\mathcal{L}_{-X_1}-q)w({V}^\pi(s-))   \mathrm{d}s
		-\int_{0}^{t\wedge T_n}e^{-qs}w'({V}^\pi(s-))\mathrm{d}{L}^\pi(s)   
		\\&+ \int_0^{t\wedge T_n}e^{-qs}w'({V}^\pi(s-)) \mathrm{d} R^{\pi}(s) + M(t \wedge T_n),
	\end{align*}
	where $M=(M(t);t\geq 0)$ is a local martingale defined in (5.8) of \cite{APWY}.
	\par Hence we derive that
	\begin{equation*}
		\begin{split}
			w(x) =&
			-\int_{0}^{t\wedge T_n}e^{-qs}  \left[ (\mathcal{L}_{-X_1}-q)w({V}^\pi(s-))-  \ell^\pi(s) (w'({V}^\pi(s-)) - 1) + r^\pi(s) (w'({V}^\pi(s-))-\beta ) \right]  \mathrm{d}s \\
			& + \int_{0}^{t\wedge T_n}e^{-qs} \ell^\pi(s)\mathrm{d}s- \beta\int_{0}^{t\wedge T_n}e^{-qs} r^\pi(s) \mathrm{d}s - M(t\wedge T_n) + e^{-q(t\wedge T_n)}w({V}^\pi(t\wedge T_n)).
		\end{split}
	\end{equation*}
	Using  the assumption \eqref{HJB-inequality}, the fact that $\ell^\pi(s) \in [0, \delta_1]$ and $r^\pi(s) \in [0, \delta_2]$ a.s.\ 
	for all $s \geq 0$, we have 
	\begin{equation} \label{w_lower}
	\begin{split}
	w(x) \geq &
	\int_{0}^{t\wedge T_n}e^{-qs} \ell^\pi(s)\mathrm{d}s -\beta\int_{0}^{t\wedge T_n}e^{-qs} r^\pi(s) \mathrm{d}s - M(t\wedge T_n)+ e^{-q(t\wedge T_n)}w({V}^\pi(t\wedge T_n)).
	\end{split}
	\end{equation}
	\par
	Note that from \eqref{v_pi_1}, and the fact that $\ell^\pi(s) \in [0, \delta_1]$ and $r^\pi(s) \in [0, \delta_2]$ a.s.\ 
	for all $s \geq 0$, we have  
	\[
	|w(x)|\leq \frac{\delta_1}{q}+\beta\frac{\delta_2}{q}+|\rho| =: K \quad\text{for $x\in(0,\infty)$.}
	\]
	
	By the compensation formula (cf.\ Corollary 4.6 of \cite{K}), $(M(t \wedge T_n):t\geq0 )$ is a zero-mean $\mathbb{P}_x$-martingale.  Now by taking expectations  in \eqref{w_lower} and letting 
	$t$ and $n$ go to infinity ($T_n\xrightarrow{n \uparrow \infty} \kappa_0^{\pi}$ $\mathbb{P}_x$-a.s.), the dominated convergence theorem gives
	\begin{align*}
		w(x) \geq 
		\mathbb{E}_x \left( \int_{0}^{\kappa_0^{\pi}}e^{-qs} \ell^\pi(s)\mathrm{d}s-\beta\int_{0}^{\kappa_0^{\pi}}e^{-qs} r^\pi(s) \mathrm{d}s + \lim_{t,n\uparrow\infty} e^{-q(t\wedge T_n)}w({V}^\pi(t\wedge T_n))\right).
	\end{align*}
	Now the proof is complete because, by \eqref{v_at_zero},
	\begin{align*}
		e^{-q(t\wedge T_n)}w({V}^\pi(t\wedge T_n)) &\geq e^{-q(t\wedge T_n)}w({V}^\pi(t\wedge T_n))  1_{\{ V^\pi (t \wedge T_n ) \leq 1/n \}}- e^{-q(t\wedge T_n)} K 1_{\{ V^\pi (t \wedge T_n ) > 1/n\}} \\ &\xrightarrow{t,n \uparrow \infty} e^{-q \kappa_0^\pi} \rho 1_{\{ \kappa_0^\pi < \infty \}}.
	\end{align*}

\end{appendix}

\end{document}